\RequirePackage{amsmath}
\documentclass[a4paper, 10pt, reqno]{amsart}
\usepackage{amssymb, url, color, pb-diagram, graphicx, amscd, pb-diagram,mathrsfs}
\usepackage[nodayofweek]{datetime}
\usepackage{enumerate,stmaryrd} 
\usepackage{euscript,enumitem}
\usepackage{verbatim}
\usepackage{graphicx}
\usepackage{times}
\usepackage{arydshln} 
\usepackage{setspace}
\usepackage[space, compress, sort]{cite}
\usepackage{empheq}
\usepackage[usenames,dvipsnames]{xcolor}
\usepackage[colorlinks=true, bookmarks=true, pdfstartview=FitH, pagebackref=true]{hyperref}

\numberwithin{equation}{section}
\theoremstyle{plain}

\newtheorem{step}{Step}

\newtheorem{thm}{Theorem}[section]
\newtheorem{theorem}[thm]{Theorem}

\newtheorem{corollary}[thm]{Corollary}

\newtheorem{lemma}[thm]{Lemma}

\newtheorem{proposition}[thm]{Proposition}

\theoremstyle{remark}

\theoremstyle{definition}

\renewcommand{\phi}{\varphi}

\newcommand{\Mbar}{\overline{M}}

\newcommand{\bR}{\mathbb{R}}
\newcommand{\bN}{\mathbb{N}}

\newcommand{\gbar}{\overline{g}}
\newcommand{\gtil}{\widetilde{g}}
\newcommand{\chibar}{\overline{\chi}}

\newcommand{\ubar}{\overline{u}}

\newcommand{\cF}{\mathcal{F}}
\newcommand{\cY}{\mathcal{Y}}
\newcommand{\cZ}{\mathcal{Z}}

\newcommand{\vtil}{\widetilde{v}}
\newcommand{\util}{\widetilde{u}}
\newcommand{\Qtil}{\widetilde{Q}}
\newcommand{\Ttil}{\widetilde{T}}
\newcommand{\Wtil}{\widetilde{W}}

\newcommand{\cC}{\mathcal{C}}

\newcommand{\Vbar}{\overline{V}}
\newcommand{\Wbar}{\overline{W}}
\newcommand{\phibar}{\overline{\phi}}

\newcommand{\phitil}{\widetilde{\phi}}
\newcommand{\Atil}{\widetilde{A}}

\newcommand{\Xbar}{\overline{X}}
\newcommand{\Deltabar}{\overline{\Delta}}

\newcommand{\ghat}{{\widehat{g}}}

\newcommand{\definedas}{\mathrel{\raise.095ex\hbox{\rm :}\mkern-5.2mu=}}
\newcommand{\ssubset}{\subset\joinrel\subset}

\newcommand{\pdiff} [2]{\frac{\partial #1}{\partial #2}}

\let\<\langle 
\let\>\rangle

\DeclareMathOperator{\tr}{tr}

\DeclareMathOperator{\vol}{Vol}
\DeclareMathOperator{\inj}{inj}

\DeclareMathOperator{\loc}{loc}
\renewcommand{\labelenumi}{\arabic{enumi}.}
\DeclareMathOperator{\supp}{supp}

\newcommand{\scal}{\mathrm{Scal}}

\newcommand{\hscal}{\widehat{\scal}}

\newcommand{\scaltil}{\widetilde{\scal}}

\renewcommand{\leq}{\leqslant}
\renewcommand{\geq}{\geqslant}

\newcommand{\Deltatil}{\widetilde{\Delta}}

\begin{document}

\title[Prescribed scalar curvature]
{Prescribed non positive scalar curvature on asymptotically hyperbolic
manifolds with application to the Lichnerowicz equation}

\begin{abstract}
We study the prescribed scalar curvature problem, namely
finding which function can be obtained as the scalar curvature of a
metric in a given conformal class. We deal with the case of
asymptotically hyperbolic manifolds and restrict ourselves to non positive
prescribed scalar curvature. Following \cite{Rauzy,DiltsMaxwell}, we
obtain a necessary and sufficient condition on the zero set of the
prescribed scalar curvature so that the problem admits a (unique)
solution.
\end{abstract}

\author[R. Gicquaud]{Romain Gicquaud}
\address[R. Gicquaud]{
  Institut Denis Poisson\\
  Universit\'e de Tours\\
  Parc de Grandmont\\ 37200 Tours \\ France}
\email{romain.gicquaud@idpoisson.fr}



\date{September 10, 2019}
\maketitle
\tableofcontents

\section{Introduction}
Given a Riemannian manifold $(M^n, g)$, the prescribed scalar curvature
problem consists in finding which function can be obtained as the scalar
curvature of a metric $\ghat$ conformal to $g$ on $M$. Namely, if $\scal$
is the scalar curvature of $g$, $\hscal$ is the scalar curvature that
one wants to prescribe and if we set $\ghat = \phi^{N-2} g$ for some
positive (unknown) function $\phi$, where $N \definedas 2n/(n-2)$,
the problem amounts to solving the following equation for $\phi$:
\begin{equation}\label{eqPrescribedScalar}
 - \frac{4(n-1)}{n-2} \Delta \phi + \scal~\phi = \hscal \phi^{N-1}.
\end{equation}
Here the exponent $N-2 = 4/(n-2)$ is chosen so that the equation has the
``nicest'' form. We refer the reader to \cite{Besse} for a discussion of
the conformal transformation laws of the curvature operators.

The equation \eqref{eqPrescribedScalar} is now well studied, at least on
compact manifolds, but getting a full understanding of the prescribed
scalar curvature problem remains a hard task. We refer the reader to
\cite{Aubin} for more details.

The aim of this paper is to extend the results of \cite{Rauzy,DiltsMaxwell}
to the context of asymptotically hyperbolic manifolds, namely to a class
of complete non-compact manifolds having sectional curvature tending to
$-1$ at infinity. See Section \ref{secPrelim} for the precise definition.
The idea of this paper is to restrict to the class of non-positive
functions $\hscal$. In this setting, the functions $\hscal$
for which Equation \eqref{eqPrescribedScalar} can be solved are
explicitely known: these are the ones whose zero set
$\cZ = \hscal^{-1}(0)$ has positive local Yamabe invariant (see Equation
\eqref{eqYamabe}). Further, in this case the solution to Equation
\eqref{eqPrescribedScalar} is unique.

In the case $\hscal < 0$, the prescribed scalar curvature equation on an
asymptotically hyperbolic manifold is by now well studied. We refer the
reader to \cite{AvilesMcowen,AnderssonChruscielFriedrich,Delay,Gicquaud}
and references therein for previous results.

Our motivation for studying this problem comes from the study of the
so called Lichnerowicz equation in general relativity for which a complete
understanding can be gained from the particular case of the prescribed
scalar curvature problem we treat in this paper. This is the topic of
Section \ref{secLich}.

The outline of this paper is as follows. Section \ref{secPrelim} is
a presentation of the function spaces that will be relevant in the paper.
Section \ref{secLocalYamabe}
introduces the local Yamabe invariant together with the local first
conformal eigenvalue which are the main ingredients to discrimiate
which scalar curvature functions $\hscal \leq 0$ can be prescribed in a
given conformal class. The main results of the paper (Theorem \ref{thmA}
and Corollary \ref{thmB}) are proven in Section \ref{secMain}. Finally,
Section \ref{secLich} is dedicated to the study of the Lichnerowicz
equation \eqref{eqLichnerowicz} on an asymptotically hyperbolic manifold.

\section{Asymptotically hyperbolic manifolds and their function spaces}
\label{secPrelim}
This section is mostly based on the monograph \cite{LeeFredholm} and on
\cite{GicquaudSakovich}. We remark that the result in this paper could
be adapted with some effort to the broader contexts described in
\cite{AllenIsenbergLeeStavrov} and
\cite{BahuaudMarsh,Bahuaud,BahuaudGicquaud,GicquaudCompactification}.

Let $\Mbar$ be a $n$-dimensional smooth manifold with boundary. We denote
by $\partial M$ the boundary of $\Mbar$ and by $M$ the interior of $\Mbar$:
$M = \Mbar \setminus \partial M$.

A \emph{defining function} for $\partial M$ is a smooth function
$\rho \geq 0$ on $\Mbar$ such that $\rho^{-1}(0) = \partial M$ and such
that $d\rho \neq 0$ on $\partial M$ (i.e. $0$ is a regular value for $\rho$).

A metric $g$ on $M$ is said to be $C^{l, \beta}$-\emph{asymptotically
hyperbolic}, where $l \geq 2$ and $\beta \in [0, 1)$, if
$\gbar \definedas \rho^2 g$ extends to a $C^{l, \beta}$ metric on $\Mbar$
such that $|d\rho|^2_{\gbar} \equiv 1$ on $\partial M$.

It can be seen that this definition is independent of the choice of the
defining function $\rho$, that the metric $g$ is complete with
sectional curvature satisfying $\sec_g = -1 + O(\rho)$ (see e.g.
\cite{MazzeoHodge}).

We fix once and for all in this section an asymptotically hyperbolic
manifold $(M, g)$ and a defining function $\rho$. Since our concern is
only Equation \eqref{eqPrescribedScalar}, we will restrict to the
definition of function spaces (i.e. spaces of real valued functions)
refering the reader to \cite{LeeFredholm,GicquaudSakovich} for the
definition of natural spaces of sections of geometric bundles. Let
$\Omega$ be an open subset of $M$. We define three classes of function
spaces:

\begin{itemize}[leftmargin=*]
\item\textsc{Weighted Sobolev spaces:} Let $0 \leq k \leq l$ be an integer,
let $1 \leq p < \infty$ be a real number, and let $\delta \in \bR$.
The weighted Sobolev space $W^{k, p}_\delta(\Omega, \bR)$ is the set of functions
$u$ such that $u \in W^{k, p}_{loc}(\Omega, \bR)$ and such that the norm
\[
\left\| u \right\|_{W^{k, p}_\delta(\Omega, \bR)} = \sum_{i=0}^k \left( \int_\Omega \left| \rho^{-\delta} \nabla^{(i)} u \right|^p_g d\mu_g\right)^{\frac{1}{p}}
\]
is finite.\\
\item\textsc{Weighted local Sobolev spaces:} Let $0 \leq k \leq l$ be an
integer, let $1 \leq p \leq \infty$ be a real number, let
$\delta \in \bR$ and $r < \inj(M, g)$ be given. The weighted local
Sobolev space $X^{k, p}_\delta(\Omega, \bR)$ is the set of functions $u$ such
that $u \in W^{k, p}_{loc}(\Omega, \bR)$ and such that the norm
\[
\left\| u \right\|_{X^{k, p}_\delta(\Omega, \bR)} = \sup_{x \in \Omega} \rho^{-\delta}(x) \left\|u\right\|_{W^{k, p}(B_r(x)\cap \Omega)}
\]
is finite. Here $B_r(x)$ denotes the ball of radius $r$ centered at $x$.\\
\item\textsc{Weighted H\"older spaces:} Let an integer $k \geq 0$ and
$0 \leq \alpha < 1$ be such that $k + \alpha \leq l + \beta$, let
$\delta \in \bR$ and $r < \inj(M, g)$ be given. The weighted H\"older
space $C^{k, \alpha}_\delta(\Omega, \bR)$ is the set of functions $u$ such
that $u \in C^{k, \alpha}_{loc}(\Omega, \bR)$ and such that the norm
\[
\left\|u\right\|_{C^{k, \alpha}_\delta(\Omega, \bR)} = \sup_{x \in \Omega} \rho^{-\delta}(x) \left\| u\right\|_{C^{k, \alpha}(B_1(x)\cap \Omega)}
\]
is finite.
\end{itemize}
We remark that we did not indicate the dependence of the norm with
respect to $r$ since different choices for $r \in (0, \inj(M, g))$
yield equivalent norms.
Weighted H\"older and Sobolev spaces are studied in great detail in
\cite{And93,GrahamLee,LeeFredholm} while weighted local Sobolev spaces
were introduced in \cite{GicquaudSakovich}. Some further properties of
these function spaces will be given in the rest of this section.\\

\noindent\textsc{Notation:} 
We choose $\chi: \bR_+ \to [0, 1]$ to be an arbitrary smooth
function such that $\chi \equiv 1$ on the interval $[0, 1]$ and
$\chi \equiv 0$ on $[2, \infty)$. For any $\rho_0 > 0$, we set
$\chibar_{\rho_0} \definedas \chi(\rho/\rho_0): \Mbar \to [0, 1]$
so that $\chibar_{\rho_0} \equiv 1$ near infinity. We
also set $\chi_{\rho_0} \definedas 1-\chi(\rho/\rho_0)$
so $\chi_{\rho_1}$ has compact support in $M$ and
$\chi_{\rho_0} + \chibar_{\rho_0} \equiv 1$.
We also define $M_{\rho_0} \definedas \rho^{-1}(0, \rho_0)$ so that
$\chibar_{\rho_0} \equiv 1$ on $M_{\rho_0}$. This notation
will be useful in the next lemma.\\

\begin{lemma}\label{lmMapping}
Given $\rho_0 > 0$ and $\delta \geq 0$, let $f \in X^{0, p}_\delta(M, \bR)$
for some $p>n/2$ and $u \in W^{1, 2}_0(M_{\rho_0}, \bR)$ be given. Then
$f u^2 \in L^1(M_{\rho_0}, \bR)$ and for any $\mu > 0$, there exists
a constant $C_\mu$ depending on $(M, g, p)$ but not on $\rho_0$ such
that
\[
\left|\int_{M_{\rho_0}} f u^2 d\mu^g\right|
 \leq \rho_0^\delta \|f\|_{X^{0,p}_\delta(M, \bR)} \left(\mu^2 \|u\|^2_{W^{1, 2}_0(M_{\rho_0}, \bR)} + C_\mu \|u\|^2_{L^2(M_{\rho_0}, \bR)}\right).
\]
\end{lemma}

\begin{proof}
We let the reader convince himself that there exists a family of extension
operators $v \mapsto \vtil$ from $W^{1, 2}_0(M_{\rho_0}, \bR)$ to
$W^{1, 2}_0(M, \bR)$ for $\rho_0$ small enough such that
$\supp(\vtil) \subset M_{2\rho_0}$ and such that there exists a constant
$\Lambda$ independent of $\rho_0$ such that
$\|\vtil\|_{W^{1, 2}_0(M, \bR)} \leq \Lambda \|v\|_{W^{1, 2}_0(M_{\rho_0}, \bR)}$
for all $v \in W^{1, 2}_0(M_{\rho_0}, \bR)$ (see e.g.
\cite[Theorem 7.25]{GilbargTrudinger} for a construction of the extension
operators on $\bR^n$).

It follows from \cite[Lemma 2.2]{LeeFredholm} that $M$ can be covered by
countably many open balls $B_i = B_r(x_i)$, where
$0 < r < \min \{\inj(M, g), 1\}$ is arbitrary, $i \in I$, such that the
cover is uniformly locally finite: for some $K \geq 1$ and for all
$x \in M$, $\# \{i \in I,~x \in B_i\} \leq K$. Let $I_0 \subset I$ be
the set of balls intersecting $M_{2\rho_0}$ :
$I_0 = \{i \in I,~ B_i \cap M_{2\rho_0} \neq \emptyset\}$.

Let $q$ be such that
\[
 \frac{1}{p} + \frac{2}{q} = 1.
\]
Since $n/2 < p < \infty$, we have $q \in (2, N)$. Let $\lambda \in (0, 1)$
be such that
\[
 \frac{1}{q} = \frac{\lambda}{2} + \frac{1-\lambda}{N},
\]
then we have, for all $i \in I_0$ and arbitrary $\mu > 0$,
\begin{align*}
\|\util\|_{L^q(B_i, \bR)}
 &\leq \|\util\|_{L^2(B_i, \bR)}^\lambda \|\util\|_{L^N(B_i, \bR)}^{1-\lambda}\\
 &\leq \lambda \mu^{-\frac{\lambda}{1-\lambda}} \|\util\|_{L^2(B_i, \bR)} + (1-\lambda) \mu \|\util\|_{L^N(B_i, \bR)}
\end{align*}
by Young's inequality. The function $t \mapsto t^2$ being convex, we have
\[
 \|\util\|^2_{L^q(B_i, \bR)} \leq \lambda \mu^{-\frac{2\lambda}{1-\lambda}} \|\util\|_{L^2(B_i, \bR)}^2 + (1-\lambda) \mu^2 \|\util\|_{L^N(B_i, \bR)}^2.
\]
Hence,
\begin{align*}
\int_M |f| \util^2 d\mu^g
 &\leq \sum_{i \in I_0} \int_{B_i} |f| \util^2 d\mu^g\\
 &\leq \sum_{i \in I_0} \|f\|_{L^p(B_i, \bR)} \|\util\|^2_{L^q(B_i, \bR)}\\
 &\leq \sum_{i \in I_0} \rho_i^\delta \|f\|_{X^{0,p}_\delta(M_{\rho_0}, \bR)} \left(\lambda \mu^{-\frac{2\lambda}{1-\lambda}} \|\util\|_{L^2(B_i, \bR)}^2 + (1-\lambda) \mu^2 \|\util\|_{L^N(B_i, \bR)}^2\right).
\end{align*}
We next claim that there exists a uniform constant $s > 0$ such that
for all balls $B_i$ and any function $v \in W^{1, 2}_0(B_i, \bR)$,
$\|u\|_{L^N(B_i, \bR)}^2 \leq s \|u\|_{W^{1, 2}_0(B_i, \bR)}^2$. This
follows from the fact that the metric has curvature bounded from above
and from below on $M$ and $r < \min \{\inj(M, g), 1\}$. As a consequence, the previous
estimate becomes
\begin{align*}
&\int_M |f| \util^2 d\mu^g\\
&\qquad\leq \left(\max_{i \in I_0} \rho_i^\delta\right) \|f\|_{X^{0,p}_\delta(M_{\rho_0}, \bR)} \left(\lambda \mu^{-\frac{2\lambda}{1-\lambda}} \sum_{i \in I_0} \|\util\|_{L^2(B_i, \bR)}^2 + s (1-\lambda) \mu^2 \sum_{i \in I_0}\|\util\|_{W^{1, 2}(B_i, \bR)}^2\right)\\
&\qquad\leq K \left(\max_{i \in I_0} \rho_i^\delta\right) \|f\|_{X^{0,p}_\delta(M_{\rho_0}, \bR)} \left(\lambda \mu^{-\frac{2\lambda}{1-\lambda}} \int_M \util^2 d\mu^g + s (1-\lambda) \mu^2 \int_M \left(|d\util|^2 + \util^2\right) d\mu^g\right)\\
&\qquad\leq \Lambda K \left(\max_{i \in I_0} \rho_i^\delta\right) \|f\|_{X^{0,p}_\delta(M_{\rho_0}, \bR)} \left(\lambda \mu^{-\frac{2\lambda}{1-\lambda}} \|u\|^2_{L^2(M_{\rho_0}, \bR)} + s (1-\lambda) \mu^2 \|u\|^2_{W^{1, 2}_0(M_{\rho_0}, \bR)}\right).
\end{align*}
The result follows by redefinig $\mu$.
\end{proof}

\begin{lemma}\label{lmCompact}
Let $\delta' > 0$, $p > n/2$ and $f \in X^{0, p}_{\delta'}(M, \bR)$ be
given. For arbitrary $\delta \in \bR$, the multiplication mapping
$u \mapsto fu$ from $X^{2, p}_\delta(M, \bR)$ to $X^{0, p}_\delta(M, \bR)$
is compact.
\end{lemma}

Note that this lemma is the extension to local Sobolev spaces of
\cite[Lemma 3.6]{LeeFredholm} and \cite[Theorem 2.3]{And93} that are left
unproven.

\begin{proof}
Let $(u_k)_{k \geq 0}$ be an arbitrary bounded sequence of elements of
$X^{2, p}_\delta(M, \bR)$. We have to show that there exists a subsequence
$(u_{\theta(k)})_{k\geq 0}$ such that $(f u_{\theta(k)})_{k\geq 0}$
converges in $X^{0, p}_\delta(M, \bR)$. Note first that the Sobolev
embedding theorem \cite[Proposition 2.3]{GicquaudSakovich} together
with the multiplication property \cite[Lemma 2.4]{GicquaudSakovich}
of the weighted local Sobolev spaces ensures that
$fu \in X^{0, p}_{\delta+\delta'}(M, \bR) \subset X^{0, p}_{\delta'}(M, \bR)$.

The proof is based on a diagonal extraction process.
We use once again the covering lemma \cite[Lemma 2.2]{LeeFredholm}.
$M$ can be covered by countably many open balls $B_i = B_r(x_i)$, $i \in \bN$.
It can be shown that replacing the supremum over all balls in the definition
of the norm defining weighted local Sobolev spaces by the supremum over all
balls $B_i$ yieds equivalent norms:
\[
\left\| u \right\|'_{X^{2, p}_\delta(M, \bR)} = \sup_{i \in \bN} \rho^{-\delta}(x_i) \left\|u\right\|_{W^{2, p}(B_r(x_i))}.
\]
This follows from the fact that $\rho$ can be chosen to be $e^{-d_g(\cdot, K)}$
with $K$ some compact set of $M$ (see e.g. \cite[Section 5]{LeeSpectrum}
and \cite{Bahuaud,BahuaudGicquaud,GicquaudCompactification}) together
with the triangle inequality for $d_g$.

We arrange the balls so that $\rho_i = \rho(x_i)$ is a decreasing function.
Note that $\rho_i \to 0$ as $i$ goes to infinity. Since the metric has
bounded geometry and using Rellich compactness theorem, we can define
inductively a sequence of (strictly) inceasing functions $\theta_i$ so
that $(u_{\theta_0(k)})_{k \geq 0}$ conveges in $L^\infty(B_0, \bR)$
and such that $(u_{\theta_{i+1}(k)})_{k \geq 0}$ is extracted from
$(u_{\theta_i(k)})_{k \geq 0}$ and converges in $L^\infty(B_{i+1}, \bR)$.
Set $\theta(k) = \theta_k(k)$ for all $k \geq 0$. The sequence $(u_{\theta(k)})_{k\geq 0}$
converges in $L^\infty_{loc}(M, \bR)$ to some function
$u \in L^\infty_{loc}(M, \bR)$ satisfying
$\left\|u\right\|_{L^\infty(B_i, \bR)} \leq C \rho^\delta$,
where $C > 0$ depends on the upper bound for
$\left(\|u_k\|_{X^{2, p}_\delta(M, \bR)}\right)_{k \geq 0}$ together with
the Sobolev constants of the embeddings
$W^{2, p}(B_i) \hookrightarrow L^\infty(M, \bR)$. Note that
$f u \in X^{0, p}_\delta(M, \bR)$. We claim that $f u_k \to f u$ in
$X^{0, p}_\delta(M, \bR)$. Fix an $\epsilon > 0$. Since
$f \in X^{0, p}_{\delta'}(M, \bR)$, there exists $i_0 > 0$ such that for
all $i \geq i_0$, we have $\|f\|_{L^P(B_i, \bR)} < \epsilon/(2C)$.
Hence, if $i \geq i_0$, we have
\begin{align*}
 &\|f u - f u_{\theta(k)}\|_{L^p(B_i, \bR)}\\
 &\qquad \leq \|f\|_{L^p(B_i, \bR)} \left( \|u\|_{L^\infty(B_i, \bR)} + \|u_{\theta(k)}\|_{L^\infty(B_i, \bR)}\right)\\
 &\qquad < \epsilon \rho^\delta(x_i).
\end{align*}
On the other hand, if $i < i_0$, then, for $k$ large enough, we have, by
construction of the function $\theta$,
\[
 \|f u - f u_{\theta(k)}\|_{L^p(B_i, \bR)} \leq \|f\|_{L^p(B_i, \bR)} \|u - u_{\theta(k)}\|_{L^\infty(B_i, \bR)} < \epsilon \rho^\delta(x_i).
\]
We have proven that $f u_k \to f u$ in $X^{0, p}_\delta(M, \bR)$. This
ends the proof of the lemma.
\end{proof}

\section{Local Yamabe invariant and first conformal eigenvalue}
\label{secLocalYamabe}
Let $(M, g)$ be a complete Riemannian manifold. For any measurable
subset $V \subset M$, we define the space
\[
 \Wbar^{k, p}(V, \bR) \definedas
  \{u \in W^{k, p}_0(M, \bR), u \equiv 0\text{ a.e. on } M \setminus V\}
\]
of Sobolev functions vanishing outside $V$. This set is obviously reduced
to $\{0\}$ if $V$ has Lebesgue measure zero. Yet the condition for
$\Wbar^{k, p}(V, \bR)$ to be non-trivial is more subtle, see for example
\cite[Chapter 6]{AdamsHedberg}. As a shorthand, for
any measurable $V$ we set
\begin{equation}\label{eqDefF}
 \cF(V) \definedas \Wbar^{1, 2}(V, \bR).
\end{equation}
An important ingredient in what follows is the following functional:
\begin{equation}\label{eqDefG}
 G_g(u) \definedas \int_M \left[\frac{4(n-1)}{n-2} |du|^2 + \scal~u^2\right] d\mu^g
\end{equation}
defined for all $u \in W^{1, 2}_0(M, \bR)$.
We also introduce, for any $u \in W^{1, 2}_0(M, \bR)$, $u \not\equiv 0$,
the Rayleigh and the Yamabe quotients:
\begin{subequations}
\begin{align}
 Q^R_g(u) & \definedas G_g(u) / \|u\|_{L^2(M, \bR)}^2,\label{eqRayleighQuotient}\\
 Q^Y_g(u) & \definedas G_g(u) / \|u\|_{L^N(M, \bR)}^2.\label{eqYamabeQuotient}
\end{align}
\end{subequations}
From these notions, we introduce the local first conformal eigenvalue
and the local Yamabe invariant of any measurable subset $V \subset M$ as
follows:
\begin{subequations}
\begin{align}
 \lambda_g(V) & \definedas \inf_{u \in \cF(V) \setminus \{0\}} Q^R_g(u)\label{eqLambda1},\\
 \cY_g(V) & \definedas \inf_{u \in \cF(V) \setminus \{0\}} Q^Y_g(u)\label{eqYamabe},
\end{align}
\end{subequations}
being understood that $\lambda_g(V) = \cY_g(V) = \infty$ if $\cF(V)$ is
reduced to $\{0\}$. We first state a lemma that will turn out useful later on:
\begin{lemma}[Asymptotic Poincar\'e inequality]\label{lmPoincare}
There exists a function $\epsilon = \epsilon(x) \in C^0(M, \bR)$, tending
to zero at infinity, such that for all $u \in W^{1, 2}_0(M, \bR)$, we have
\[
 \int_M \left[\frac{4(n-1)}{n-2} |du|^2 + \scal~u^2\right]d\mu^g
 \geq
 \int_M \left(\frac{n-1}{n-2} + \epsilon\right)u^2 d\mu^g.
\]
\end{lemma}

\begin{proof}
 Let $\rho > 0$ be a defining function for $\partial M$. Given $u \in W^{1, 2}_0(M, \bR)$,
we set $v = \rho^{-\delta} u$ for some $\delta$ to be chosen later and compute
\begin{align*}
\int_M |du|^2 d\mu^g
&= \int_M |d(\rho^\delta v)|^2 d\mu^g\\
&= \int_M |\delta \rho^{\delta-1} v d\rho + \rho^\delta dv|^2 d\mu^g\\
&= \int_M \left[\rho^{2\delta} |dv|^2 + 2 \delta \rho^{2\delta-1} \<d\rho, v dv\> + \rho^{2 \delta-2} v^2 |d\rho|^2\right] d\mu^g\\
&= \int_M \left[\rho^{2\delta} |dv|^2 + \frac{1}{2} \<d\rho^{2\delta}, d(v^2)\> + \rho^{2 \delta-2} v^2 |d\rho|^2\right] d\mu^g\\
&= \int_M \left[\rho^{2\delta} |dv|^2 - \frac{v^2}{2} \Delta (\rho^{2\delta}) + \rho^{2 \delta-2} v^2 |d\rho|^2\right] d\mu^g\\
&= \int_M \left[\rho^{2\delta} |dv|^2 - \rho^{2\delta-2} \left(\delta \rho \Delta \rho + \delta(\delta-1) |d\rho|^2\right)v^2\right] d\mu^g.
\end{align*}
The conformal transformation law of the Laplacian gives
\[
 \Delta \rho = \rho^2\left(\Deltabar \rho - (n-2) \frac{|d\rho|_{\gbar}^2}{\rho}\right),
\]
where $\Deltabar$ denotes the Laplacian associated to the metric $\gbar$.
Since $|d\rho|_{\gbar}^2 = 1 + o(1)$ and $\Deltabar \rho = O(1)$, we have
\[
 \Delta \rho = -(n-2) \rho + o(\rho).
\]
Further,
\[
 |d\rho|^2 = \rho^2 |d\rho|^2_{\gbar} = \rho^2 + o(\rho^2).
\]
As a consequence, we have
\[
\rho^{2\delta-2} \left(\delta \rho \Delta \rho + \delta(\delta-1) |d\rho|^2\right)
= \rho^{2\delta} \left(-(n-2) \delta + \delta(\delta-1) + \epsilon_0\right),
\]
where $\epsilon_0 = o(1)$ in a neighborhood of infinity. Choosing $\delta = -\frac{n-2}{2}$,
we get
\[
\rho^{2\delta-2} \left(\delta \rho \Delta \rho + \delta(\delta-1) |d\rho|^2\right) = \rho^{2\delta} \left[\frac{(n-1)^2}{4} + \epsilon_0\right].
\]
Thus,
\begin{equation}\label{eqPoincare}
\int_M |du|^2 d\mu^g
= \int_M \left[\rho^{2\delta} |dv|^2 + \left(\frac{(n-1)^2}{4} + \epsilon_0\right)u^2\right] d\mu^g.
\end{equation}
Finally,
\begin{align}
\int_M \left[\frac{4(n-1)}{n-2} |du|^2 + \scal~u^2\right]d\mu^g
&= \int_M \left[\rho^{2\delta} |dv|^2 + \left(\frac{(n-1)^2}{4} + \scal + \epsilon_0\right)u^2\right] d\mu^g\nonumber\\
&= \int_M \left[\rho^{2\delta} |dv|^2 + \left(\frac{(n-1)^2}{4} - n(n-1) + \epsilon\right)u^2\right] d\mu^g\nonumber\\
&= \int_M \left[\rho^{2\delta} |dv|^2 + \left(\frac{n-1}{n-2} + \epsilon\right)u^2\right] d\mu^g\label{eqConvex}
\end{align}
where $\epsilon = \epsilon_0 + \scal - n(n-1) = o(1)$ near infinity.
\end{proof}

\begin{proposition}\label{propSemiContinuity}
The functional $G_g$ defined in \eqref{eqDefG}
is sequentially weakly lower semi-continuous on $W^{1, 2}_0(M, \bR)$.
Namely, for every weakly converging sequence $(u_k)_k$,
$\displaystyle
 u_k \rightharpoonup u_\infty,
$
we have
$\displaystyle
 \liminf_{k\to \infty} G_g(u_k) \geq G_g(u_\infty).
$
\end{proposition}

\begin{proof}
 We use formula \eqref{eqConvex} which shows that
\[
u \mapsto G_g(u) - \int_M \epsilon u^2 d\mu^g
 = \int_M \left[ |\rho^\delta d(\rho^{-\delta} u)|^2 + \frac{n-1}{n-2} u^2\right] d\mu^g
\]
is convex. Since it is also strongly continuous, it is
weakly lower semi-continuous. So we just have to prove that
\[
 H(u) \definedas \int_M \epsilon u^2 d\mu^g
\]
is sequentially weakly continuous. Let $(u_k)_{k \geq 0}$ be a sequence of elements
of $W^{1, 2}_0(M, \bR)$ that converges weakly to $u_\infty \in W^{1, 2}_0(M, \bR)$.
We assume, by contradiction that $H(u_k)$ does
not converges to $H(u_\infty)$. Upon extracting a subsequence, we can assume
that there exists $\mu_0 > 0$ such that $|H(u_k) - H(u_\infty)| > \mu_0$
for all $k \geq 0$. Since for every bounded set $\Omega$, the map
\[
 \begin{array}{ccc}
  W^{1, 2}_0(M, \bR) & \to & L^2(\Omega, \bR)\\
  u & \mapsto & u\vert_{\Omega}
 \end{array}
\]
is compact, and $(u_k)_k$ is bounded in $W^{1, 2}_0(M, \bR)$ (hence in $L^2(M, \bR)$),
we can assume further that $\sqrt{\epsilon_\pm} u_k$ converges strongly
in $L^2(M, \bR)$, where
\[
 \epsilon_+ = \max \{0, \epsilon\}, \quad \epsilon_- = -\min \{0, \epsilon\}.
\]
However, for every $v \in L^2(M, \bR)$, we have, due to the weak convergence of $(u_k)_k$,
\[
 \int_M v (\lim_{k\to\infty} \epsilon_\pm u_k) d\mu^g = \lim_{k\to\infty} \int_M v \epsilon_\pm u_k d\mu^g = \int_M v \epsilon_\pm u_\infty d\mu^g,
\]
so
\[
 \lim_{k\to\infty} \epsilon_\pm u_k = \epsilon_{\pm} u_\infty,
\]
where all limits of functions are understood in $L^2(M, \bR)$.
On the other hand,
\begin{align*}
\lim_{k \to \infty} H(u_k)
 &= \lim_{k \to \infty} \left(\int_M \epsilon_+ u_k^2 d\mu^g - \int_M \epsilon_- u_k^2 d\mu^g\right)\\
 &= \int_M \epsilon_+ u_\infty^2 d\mu^g - \int_M \epsilon_- u_\infty^2 d\mu^g\\
 &= H(u_\infty),
\end{align*}
contradicting the fact that $|H(u_k) - H(u_\infty)| > \mu_0$.
\end{proof}

\begin{proposition}\label{propYamabe}
Given any measurable set $V \subset M$, $\lambda_g(V)$ and $\cY_g(V)$ have
the same sign (i.e. they are either both positive, both negative or both
zero).
\end{proposition}

\begin{proof}
We first remark that
\[
 \cY_g(V) < 0 \Leftrightarrow \exists u \in \cF(V), G_g(u) < 0 \Leftrightarrow \lambda_g(V) < 0.
\]
Next assume that $\cY_g(V) = 0$. We are to prove that $\lambda_g(V) = 0$.
There exists a sequence $(u_k)_{k \geq 0}$ of functions belonging to
$\cF(V)$ such that $\|u_k\|_{L^N(M, \bR)} = 1$, $0 \leq G_g(u_k)$,
$G_g(u_k) \to O$. From the (global) Sobolev embedding theorem
\cite[Lemma 3.6]{LeeFredholm}, there exists a constant $s > 0$ such that
\begin{equation}\label{eqSobolevAH}
 \|u\|_{L^N(M, \bR)}^2 \leq s \|du\|_{L^2(M, \bR)}^2
\end{equation}
for all $u \in W^{1, 2}_0(M, \bR)$. In particular, we have that
$\|du_k\|_{L^2(M, \bR)}^2 \geq s$. For $k$ large enough, we have that
\[
 G_g(u_k) = \int_M\left[\frac{4(n-1)}{n-2} |du_k|^2 + \scal~u_k^2 \right]d\mu^g \leq \frac{2(n-1)}{n-2} s.
\]
As a consequence,
\[
 -\|\scal\|_{L^\infty(M, \bR)} \|u_k\|_{L^2(M, \bR)}^2 \leq -\frac{2(n-1)}{n-2} s.
\]
This shows that $\|u_k\|_{L^2(M, \bR)}$ is bounded from below by a positive
constant. Hence, $Q^R_g(u_k) \to 0$. Since $\lambda_g(V) \geq 0$ from the
first part of the proof, we have that $\lambda_g(V) = 0$.

We now prove that, conversely, if $\lambda_g(V) = 0$, we have
$\cY_g(V) = 0$. We select a sequence of functions $u_k \in \cF(V)$ such
that $G_g(u_k) \geq 0$ and $G_g(u_k) \to_{k \to \infty} 0$.
From Lemma \ref{lmPoincare}, we have that there exists a compact subset
$K \ssubset M$ and a constant $C > 0$ such that
\[
 \int_{M \setminus K} u^2 d\mu^g \leq G_g(u) + C \int_K u^2 d\mu^g.
\]
Indeed, one can choose for example
\[
 K = \epsilon^{-1}(-\infty, -1/(n-2)) \quad\text{and}\quad C = \left\|\epsilon + \frac{n-1}{n-2}\right\|_{L^\infty(K, \bR)}.
\]
As a consequence,
\[
 1 = \int_M u_k^2 d\mu^g \leq G_g(u_k) + (C+1) \int_{\Omega} u_k^2 d\mu^g.
\]
This shows in particular that, for $k$ large enough, $\|u_k\|_{L^2(K, \bR)}$
is bounded from below by a positive constant. Now we have that
\[
 \|u_k\|_{L^2(K, \bR)} \leq \|u_k\|_{L^N(K, \bR)} \vol_g(K)^{1/n} \leq \|u_k\|_{L^N(M, \bR)} \vol_g(K)^{1/n}
\]
which proves that $\|u_k\|_{L^N(M, \bR)}$ is bounded from below (for
$k$ large enough). We then conclude that $Q^Y_g(u_k) \to 0$. Thus,
$\cY_g(V) = 0$.

Finally remark that, from everything we have proven before, $\cY_g(V) > 0$
iff $\lambda_g(V) > 0$. This ends the proof of the proposition.
\end{proof}

The interest for working with $\cY_g(V)$ instead of $\lambda_g(V)$ comes
from the following result:

\begin{proposition}\label{propConfInvariance}
Assume that $g$ and $h$ are two conformally related metrics,
$h = \phi^{N-2} g$, for some $\phi \in X^{2, p}_0(M, \bR)$,
$p > n/2$, $\phi > 0$ with $\phi^{-1} \in X^{2, p}_0(M, \bR)$,
then for any measurable $V$ we have
\[
 \cY_g(V) = \cY_h(V).
\]
\end{proposition}

\begin{proof}
The proof is a simple calculation. Given any $u \in W^{1, 2}_0(M, \bR)$, we
have
\begin{align*}
G_h(u)
 &= \int_M \left[\frac{4(n-1)}{n-2} |du|_h^2 + \scal^h~u^2\right] d\mu^h\\
 &= \int_M \left[\frac{4(n-1)}{n-2} \phi^{2-N} |du|_g^2 + \left(-\frac{4(n-1)}{n-2} \Delta^g \phi + \scal^g~\phi\right) \phi^{1-N} u^2\right] \phi^N d\mu^g\\
 &= \int_M \left[\frac{4(n-1)}{n-2} \phi^2 |du|_g^2 + \left(-\frac{4(n-1)}{n-2} \Delta^g \phi + \scal^g~\phi\right) \phi u^2\right] d\mu^g\\
 &= \int_M \left[\frac{4(n-1)}{n-2} \left(\phi^2 |du|_g^2 - (\phi \Delta^g \phi) u^2\right) + \scal^g~(\phi u)^2\right] d\mu^g\\
 &= \int_M \left[\frac{4(n-1)}{n-2} \left(\phi^2 |du|_g^2 + \<d\phi, d(\phi u^2)\>_g\right) + \scal^g~(\phi u)^2\right] d\mu^g\\
 &= \int_M \left[\frac{4(n-1)}{n-2} \left(\phi^2 |du|_g^2 + u^2 |d\phi|^2_g + 2 \<\phi d\phi, u du\>_g\right) + \scal^g~(\phi u)^2\right] d\mu^g\\
 &= \int_M \left[\frac{4(n-1)}{n-2} |d(\phi u)|_g^2 + \scal^g~(\phi u)^2\right] d\mu^g\\
 &= G_g(\phi u).
\end{align*}
Similarly,
\[
 \|u\|_{L^N_h} = \left(\int_M u^N d\mu^h\right)^{1/N} = \left(\int_M u^N \phi^N d\mu^g\right)^{1/N} = \|\phi u\|_{L^N_g}.
\]
So
\[
 Q^\cY_h(u) = Q^\cY_g(\phi u).
\]
Since $\phi$ is bounded away from zero, multiplication by $\phi$ defines
an automorphism of $\cF(V)$. Indeed, $\phi, \phi^{-1} \in L^\infty(M, \bR)$,
so for any $u \in W^{1, 2}_0(M, \bR)$,
we have $\phi u \in L^2(M, \bR)$ and $\phi du \in L^2(M, T^*M)$. Since
$d(\phi u) = \phi du + u d\phi$, we only have to show that
$u d\phi \in L^2(M, T^*M)$. This follows at once from Lemma \ref{lmMapping}
applied to $f \equiv |d\phi|^2$.

Hence,
\[
 \cY_g(V) = \inf_{u \in \cF(V)} Q^\cY_g(u) = \inf_{u \in \cF(V)} Q^\cY_g(\phi u) = \inf_{u \in \cF(V)} Q^\cY_h(u) = \cY_h(V).
\]
\end{proof}

\section{Prescribing non-positive scalar curvature on AH manifolds}
\label{secMain}
In thi section, we prove the main result of this paper:
\begin{theorem}\label{thmA}
Let $(M, g)$ be a $C^{l, \beta}$-asymptotically hyperbolic manifold with
$l \in \bN$, $l \geq 2$, $\beta \in [0, 1)$. Assume given
$\hscal \leq 0$ such that $\hscal - \scal \in X^{0, p}_\delta(M, \bR)$,
with $p > \max\{2, \frac{n}{2}\}$ and $\delta \in (0, n)$. The following
assertions are equivalent:
\begin{enumerate}
\renewcommand{\theenumi}{\roman{enumi}}
\renewcommand{\labelenumi}{\theenumi.}
 \item\label{itExists} there exists a positive function $\phi > 0$,
$\phi-1 \in X^{2, p}_\delta(M, \bR)$, such that the metric
$\ghat \definedas \phi^{N-2} g$ has scalar curvature $\hscal$,
 \item\label{itPositive} the set 
\begin{equation}\label{eqDefZ}
 \cZ \definedas \{x \in M, \hscal(x) = 0\}
\end{equation}
has positive Yamabe invariant: $\cY_g(\cZ) > 0$
\end{enumerate}
Further, the function $\phi$ is then unique.
\end{theorem}
If $\hscal$ enjoys further decay properties, Theorem \ref{thmA} can be
improved:
\begin{corollary}\label{thmB}
Under the assumptions of the previous theorem, if
$\hscal - \scal \in X^{k, p}_\delta(M, \bR)$ (resp.
$\hscal - \scal \in C^{k, \alpha}_\delta(M, \bR)$)
for some $k \in \bN$, $k \leq l$, and $\delta \in (0, n)$
(resp. $k \in \bN$, $\alpha \in (0, 1)$, $k+\alpha \leq l+\beta$)
we have
$\phi-1 \in X^{k, \alpha}_\delta(M, \bR)$ (resp.
$\phi-1 \in C^{k, \alpha}_\delta(M, \bR)$).
\end{corollary}

The proof of the corollary follows from an elliptic regularity argument
applied to the function $\phi-1$ obtained in Theorem \ref{thmA}. See e.g.
\cite{GicquaudSakovich}.

We first prove that \eqref{itExists} $\Rightarrow$ \eqref{itPositive} in
Theorem \ref{thmA} because this implication is much simpler than its converse.

\begin{proof}[Proof of \eqref{itExists} $\Rightarrow$ \eqref{itPositive}
in Theorem \ref{thmA}]
Let $\phi$ be the solution to Equation \eqref{eqPrescribedScalar} given
by \eqref{itExists} and $\ghat = \phi^{N-2} g$. Then, for any $u \in \cF(\cZ)$,
we have, from the proof of Proposition \ref{propConfInvariance},
\begin{align*}
G_g(\phi u)
 &= G_{\ghat}(u)\\
 &= \int_M \left[\frac{4(n-1)}{n-2} |du|_{\ghat}^2 + \hscal~u^2\right] d\mu^{\ghat}\\
 &= \frac{4(n-1)}{n-2} \int_M |du|_{\ghat}^2 d\mu^{\ghat}\\
 &= \frac{4(n-1)}{n-2} \int_M \phi^2 |du|_g^2 d\mu^g\\
 & \geq \frac{4(n-1)}{n-2} \left(\min_M \phi\right)^2 \int_M |du|_g^2 d\mu^g
\end{align*}
From the (global) Sobolev embedding theorem \cite[Lemma 3.6]{LeeFredholm},
we have that for some constant $s > 0$ (independent of $u$),
\begin{align*}
 G_g(\phi u)
 & \geq s \left(\min_M \phi\right)^2 \left(\int_M |u|^N d\mu^g\right)^{2/N}\\
 & \geq s \left(\frac{\min_M \phi}{\max_M \phi}\right)^2 \left(\int_M |\phi u|^N d\mu^g\right)^{2/N}.
\end{align*}
We conclude that
\[
 \cY_g(\cZ) = \inf_{u \in \cF(\cZ)} \frac{G_g(\phi u)}{\|\phi u\|_{L^N(M, \bR)}^{2/N}} \geq s \left(\frac{\min_M \phi}{\max_M \phi}\right)^2 > 0.
\]

\end{proof}

Proving the converse statement, namely \eqref{itPositive} $\Rightarrow$
\eqref{itExists}, is more complicated. It will be carried in several steps.
We first state a maximum principle for Equation \eqref{eqPrescribedScalar}
that will be of constant use in the sequel.
\begin{proposition}\label{propMaximumPrinciple}
Let $\hscal_1, \hscal_2 \in L^p_{\loc}(M, \bR)$ be two given functions
such that $\hscal_1 \leq \hscal_2 \leq 0$. Assume that there exists two
positive functions $\phi_1, \phi_2 \in W^{2, p}_{\loc}(M, \bR)$ solving
the prescribed curvature equation \eqref{eqPrescribedScalar} with
$\hscal \equiv \hscal_1$ and $\hscal \equiv \hscal_2$ respectively:
\[
\left\lbrace
\begin{aligned}
 - \frac{4(n-1)}{n-2} \Delta \phi_1 + \scal~\phi_1 &= \hscal_1~\phi_1^{N-1},\\
 - \frac{4(n-1)}{n-2} \Delta \phi_2 + \scal~\phi_2 &= \hscal_2~\phi_2^{N-1}.
\end{aligned}
\right.
\]
If further, $\phi_1, \phi_2 \to 1$ at infinity, then $\phi_1 \leq \phi_2$.
\end{proposition}

\begin{proof}
Since $\phi_1$ and $\phi_2$ are positive functions, we can set
$\psi_1 = \log(\phi_1)$ and $\psi_2 = \log(\phi_2)$. Both $\phi_1$
and $\phi_2$ are positive and continuous. As a consequence, we have
$\psi_1, \psi_2 \in W^{2, p}_{\loc}$. It is straightforward to check that
they satisfy
\[
\left\lbrace
\begin{aligned}
 - \frac{4(n-1)}{n-2}\left( \Delta \psi_1 + |d\psi_1|^2\right) + \scal &= \hscal_1~e^{(N-2) \psi_1},\\
 - \frac{4(n-1)}{n-2}\left( \Delta \psi_2 + |d\psi_2|^2\right) + \scal &= \hscal_2~e^{(N-2) \psi_2}.
\end{aligned}
\right.
\]
Subtracting these equations, we obtain
\begin{align*}
 & - \frac{4(n-1)}{n-2}\left[ \Delta (\psi_1 - \psi_2) + \<d(\psi_1+\psi_2), d(\psi_1-\psi_2)\right]\\
 &\qquad\qquad = \hscal_2 \left(e^{(N-2) \psi_1} - e^{(N-2) \psi_2}\right) + \left(\hscal_1 - \hscal_2\right) e^{(N-2)\psi_1}\\
 &\qquad\qquad = (N-2) \hscal_2 \left(\int_0^1 e^{(N-2)(\lambda \psi_1 + (1-\lambda)\psi_2)} d\lambda\right) (\psi_1 - \psi_2) + \left(\hscal_1 - \hscal_2\right) e^{(N-2)\psi_1}.
\end{align*}
Setting $\xi = \psi_1 - \psi_2$ and
$\displaystyle f = \int_0^1 e^{(N-2)(\lambda \psi_1 + (1-\lambda)\psi_2)} d\lambda$,
we get the following equation for $\xi$:
\[
- \frac{4(n-1)}{n-2}\left[ \Delta \xi + \<d(\psi_1+\psi_2), d\xi\>\right] - (N-2) f \hscal_2 \xi
 = \left(\hscal_1 - \hscal_2\right) e^{(N-2)\psi_1}.
\]
Since $f\scal_2 \leq 0$, we can apply the maximum principle
\cite[Theorem 3.1]{TrudingerMeasurable}
on larger and larger domains $\Omega_k$ such that
$\bigcup_k \Omega_k = M$ and get that
\[
 \sup_M \xi \leq \liminf_{k \to \infty}~\sup_{\Omega_k} \xi = 0
\]
because $\psi_1, \psi_2 \to 0$ at infinity.
\end{proof}
As a consequence, we immediately get that (under the assumptions of
Theorem \ref{thmA}) if the solution $\phi$ to \eqref{eqPrescribedScalar}
exists, it is unique.

We now turn our attention to a decay estimate of the solution $\phi$
to the equation \eqref{eqPrescribedScalar}. Following \cite{GicquaudSakovich},
for any $k \in \bN$ and $p \in [1, \infty)$, we define the subspace

\[
 X^{k, p}_{0^+}(M, \bR) \definedas \{u \in X^{k, p}_0(M, \bR),~ \left\|u\right\|_{W^{k, p}(B_r(x))} = o(1)\}
\]
of $X^{k,p}_0(M, \bR)$,
where $o(1)$ refers to a quantity that goes to zero as $x$ goes to
infinity. As indicated in \cite{GicquaudSakovich}, this space is the closure
of the space $C^{l, \beta}_c(M, \bR)$ of compactly supported $C^{l, \beta}$-functions
in the space $X^{k, p}_0(M, \bR)$. Assuming that $kp > n$, it is easily seen
that $X^{k, p}_{0^+}(M, \bR)$ is a (non unital) Banach subalgebra of
$X^{k, p}_0(M, \bR)$.

Note also that the $X^{k, p}_0$-norm is weaker than the $W^{k, p}_0$-norm.
Hence, $W^{k, p}_0(M, \bR) \subset X^{k, p}_{0^+}(M, \bR)$.

\begin{proposition}\label{propImprovedDecay}
Assume that the assumptions of Theorem \ref{thmA} are fulfilled and that
$\phi$ is a positive solution to Equation \eqref{eqPrescribedScalar} such
that $\phi-1 \in X^{2, p}_{0^+}(M, \bR)$. Then
$\phi-1 \in X^{2, p}_{\delta}(M, \bR)$.
\end{proposition}

\begin{proof}
We rewrite \eqref{eqPrescribedScalar} as follows:
\begin{equation}\label{eq3}
 -\frac{4(n-1)}{n-2} \Delta \phi + \scal \left(\phi - \phi^{N-1}\right) = \left(\hscal - \scal\right) \phi^{N-1}.
\end{equation}
Since, by assumption, $\hscal - \scal \in X^{0, p}_\delta(M, \bR)$, the
right-hand side of \eqref{eq3} belongs to $X^{0, p}_\delta(M, \bR)$.
We can write $\phi^{N-1}-\phi = (N-2 + h(\phi)) (\phi-1)$ where $h$ is
an analytic function of $\phi$ such that $h(1) = 0$. Equation \eqref{eq3}
then becomes
\begin{equation}\label{eq4}
 -\frac{4(n-1)}{n-2} \Delta (\phi-1) + \scal(N-2 + h(\phi)) (\phi-1) = \left(\hscal - \scal\right) \phi^{N-1}.
\end{equation}
It follows from Lemma \ref{lmCompact} (where the assumption $f \in X^{0, p}_{\delta'}(M, \bR)$
can be replaced by $f \in X^{0, p}_{0^+}(M, \bR)$ without modification of
the proof) that the operator
\[
 \Phi: u \mapsto -\frac{4(n-1)}{n-2} \Delta u + \scal(N-2 + h(\phi)) u
\]
from $X^{2, p}_\delta(M, \bR)$ to $X^{0, p}_\delta(M, \bR)$ is Fredholm
for $\delta \in (-1, n)$ being a compact perturbation of
\[
 \Phi_0: u \mapsto \frac{4(n-1)}{n-2} \left[-\Delta u + n u\right]
\]
(see e.g. \cite[Section 3]{LeeSpectrum} and \cite[Chapter 7]{LeeFredholm}
for the study of the operator $u \mapsto -\Delta u + nu$).

From \cite[Corollary A.6]{GicquaudSakovich}, for any $\delta' \in (-1, n)$,
there exist continuous operators $\Qtil$ and $\Ttil$,
\begin{align*}
\Qtil &: X_{\delta'}^{0,p}(M,\bR)\rightarrow X_{\delta'}^{2,p}(M,\bR),\\
\Ttil &: X_{\delta'}^{1,p}(M,\bR)\rightarrow X_{\delta''}^{2,p}(M,\bR),
\end{align*}
where $\delta''$ is such that $\delta'\leq\delta''\leq\delta'+1$,
$\delta'' \in (-1, n)$, satisfying  
\[
\Qtil \Phi_0 v= v + \Ttil v 
\]
for any $v \in X^{2, p}_{\delta'}(M, \bR)$ supported in $M_{\rho_0}$ with
$\rho_0$ small enough. Thus,
\begin{equation}\label{eqParametrix0}
\Qtil \Phi v= v + \Qtil (\Phi - \Phi_0) v + \Ttil v.
\end{equation}
Now remark that $\Phi - \Phi_0$ is a multiplication operator by some
function belonging to $X^{0, p}_{0^+}(M, \bR)$. As a consequence, upon
diminishing $\rho_0$ we can assume that
\[
 \left\|\Qtil (\Phi - \Phi_0) v\right\|_{X^{2, p}_{\delta'}(M, \bR)} \leq \frac{1}{2} \left\|v\right\|_{X^{2, p}_{\delta'}(M, \bR)}
\]
for any $v \in X^{2, p}_{\delta'}(M, \bR)$ supported in $M_{\rho_0}$.
Note that the set of such functions form a Banach subspace
$\Xbar^{2, p}_{\delta'}(M_{\rho_0}, \bR)$ of
$X^{2, p}_{\delta'}(M, \bR)$. From our assumption, the operator
\[
 R: v \mapsto v + \Qtil (\Phi - \Phi_0) v
\]
is invertible on $\Xbar^{2, p}_{\delta'}(M_{\rho_0}, \bR)$. Hence,
setting $Q = R^{-1} \Qtil$ and $T = R^{-1} \Ttil$, we have
\footnote{Remark that we are slightly sloppy here since the image of
$\Qtil$ is $X_{\delta'}^{2,p}(M,\bR)$ while $R^{-1}$ has range
$\Xbar_{\delta'}^{0,p}(M_{\rho_0},\bR)$. This can be circumvented by
setting $Q = R^{-1} C \Qtil$ with $C$ a multiplication operator by some
well chosen cutoff function. Details are left to the reader.}
\begin{equation}\label{eqParametrix1}
Q \Phi v= v + T v
\end{equation}
for any $v \in X_0$ and
\begin{align*}
Q &: X_{\delta'}^{0,p}(M_{\rho_0},\bR)\rightarrow \Xbar_{\delta'}^{2,p}(M,\bR),\\
T &: X_{\delta'}^{1,p}(M_{\rho_0},\bR)\rightarrow \Xbar_{\delta''}^{2,p}(M,\bR),
\end{align*}
where $\Xbar^{k, p}_{\delta'}(M_{\rho_0}, \bR)$ is the subspace of
$X^{k, p}_{\delta'}(M, \bR)$ vanishing outside $M_{\rho_0}$.
We now set $v = \chibar_{\rho_0} (\phi-1)$. Since $v$ agrees with
$\phi-1$ on $M_{\rho_0}$, we have to show that $v \in X^{2, p}_\delta(M, \bR)$.
From \eqref{eq4}, we have
$\Phi v = \left(\hscal - \scal\right) \phi^{N-1}$ on $M_{\rho_0}$ so
$\Phi v \in X^{2, p}_\delta(M, \bR)$. The strategy is to apply inductively
\eqref{eqParametrix1}. Indeed, if we know that $v \in X^{2, p}_{\delta'}(M, \bR)$
for some $\delta' < \delta$ ($\delta' \in (-1, n)$), we have
\[
 v = Q \Phi v - T v
\]
so $v \in X^{2, p}_{\delta^{(3)}}(M, \bR)$, with
$\delta^{(3)} = \min\{\delta, \delta''\}$ where $\delta'' \in (-1, n)$ is
such that $\delta' < \delta'' \leq \delta'+1$. After a finite number of
steps, we obtain that $v \in X^{2, p}_\delta(M, \bR)$.
\end{proof}

The first two steps reduce the proof of \eqref{itPositive} $\Rightarrow$
\eqref{itExists} to a nice particular case, namely $\hscal \in L^\infty(M, \bR)$
and $\hscal - \scal$ compactly supported. Existence of a solution $\phi$
to \eqref{eqPrescribedScalar} will then be obtained by a variational
argument.

\begin{step}[Reduction to $\hscal \in L^\infty$]\label{stepReduction}
It suffices to prove that \eqref{itPositive} $\Rightarrow$ \eqref{itExists}
in Theorem \ref{thmA} assuming further that $\hscal \in L^\infty(M, \bR)$.
\end{step}

\begin{proof}
Assume given $\hscal$ satisfying the assumptions of the theorem.
We set $\hscal_k = \max\{\hscal, -k\}$ for all $k \geq n(n-1)$ so that
each $\hscal_k$ belongs to $L^\infty(M, \bR)$. Note that the zero set of
any $\hscal_k$ is the same as that of $\hscal$, namely $\cZ$. Taking for
granted that the theorem is valid for any prescribed
$\hscal \in L^\infty(M, \bR)$, we get a sequence of solutions $\phi_k$,
$\phi_k-1 \in X^{2, p}_\delta(M, \bR)$ to \eqref{eqPrescribedScalar}
with $\hscal$ replaced by $\hscal_k$. Since we have, for all $k$, that
$\hscal_{k+1} \leq \hscal_k$, we get from the maximum principle
(Proposition \ref{propMaximumPrinciple}) that $\phi_{k+1} \leq \phi_k$.
The functions $\phi_k$ solve
\begin{equation}\label{eqPSCk}
-\frac{4(n-1)}{n-2} \Delta \phi_k + \scal~\phi_k = \hscal_k~\phi_k^{N-1}
\end{equation}
and the right-hand side is bounded in $L^p_{\loc}(M, \bR)$:
\[
 0 \geq \hscal_k~\phi_k^{N-1} \geq \hscal~\phi_{n(n-1)}^{N-1}.
\]
So we get
that the sequence $(\phi_k)_k$ is bounded in $W^{2, p}_{\loc}(M, \bR)$.
Since the functions $\phi_k$ are also positive, it follows
that $(\phi_k)_k$ converges
uniformly on any compact subset $K \ssubset M$ to some continuous
function $\phi_\infty$, $0 \leq \phi_\infty \leq \phi_{n(n-1)}$. Hence,
$\hscal_k \phi_k^{N-1} \to \hscal \phi_\infty$ in $L^p_{\loc}(M, \bR)$.
This allows to pass to the limit in Equation \eqref{eqPSCk}: $\phi_\infty$
satisfies, at least in a weak sense
\[
 -\frac{4(n-1)}{n-2} \Delta \phi_\infty + \scal~\phi_\infty = \hscal~\phi_\infty^{N-1}.
\]

The difficulty consists in proving that $\phi_\infty \not\equiv 0$.
Let $\epsilon > 0$ be a small enough regular value of the defining
function $\rho$, so that the subset
$\Sigma_\epsilon = \rho^{-1}(\epsilon)$ is a smooth hypersurface of $M$.
Assume for the moment that there exists two positive functions $\phibar_0$
and $\phibar_1$ on $M_\epsilon = \rho^{-1}(0, \epsilon)$ such that
\begin{equation}\label{eqDefPhibar}
\begin{aligned}
 &-\frac{4(n-1)}{n-2} \Delta \phibar_0 + \left(\scal-\hscal\right)~\phibar_0 = 0,\quad \phibar_0 = 0~\quad\text{on}~\Sigma_\epsilon,\quad \phibar_0 \to 1~\text{at infinity},\\
 &-\frac{4(n-1)}{n-2} \Delta \phibar_1 + \left(\scal-\hscal\right)~\phibar_1 = 0,\quad \phibar_1 = 1~\quad\text{on}~\Sigma_\epsilon,\quad \phibar_1 \to 0~\text{at infinity}.
\end{aligned}
\end{equation}
It follows by a modification of the proof of Proposition \ref{propImprovedDecay}
that $\phibar_0-1, \phibar_1 \in X^{2, p}_\delta(M_\epsilon, \bR)$ as long as
$\delta \in (0, n-1)$.
Let $\lambda \in (0, 1)$ be arbitrary. We claim that if $\Lambda > 1$ is
large enough, $\phi_- \definedas \lambda \phibar_0 - \Lambda \phibar_1$
is a subsolution to \eqref{eqPrescribedScalar} (and hence to
\eqref{eqPSCk}) wherever it is positive. This follows at once by noticing
that $\phi_-$ satisfies 
\[
 -\frac{4(n-1)}{n-2} \Delta \phi_- + \scal~\phi_- = \hscal~\phi_-
\]
and that the right-hand side is less than or equal to $\hscal~\phi_-^{N-1}$
as soon as $\phi_- \leq 1$. But the set of points where $\lambda \phibar_0$
is greater than $1$ is compact. Since $\phibar_1$ is positive, by choosing
$\Lambda$ large enough, we can ensure that $\phi_- \leq 1$. Applying the
maximum principle (Proposition \ref{propMaximumPrinciple}) to \eqref{eqPSCk}
over the region $M_\epsilon \cap \{\phi_- \geq 0\}$, we conclude that
$\phi_k \geq \phi_-$ for all $k$. Hence, passing to the limit
$\phi_\infty \geq \phi_-$.

We have proven that $\phi_\infty \not\equiv 0$ but we get even more:
\[
 \liminf_{x \to \infty} \phi_\infty(x) \geq \liminf_{x \to \infty} \phi_- = \lambda.
\]
Since this holds for any choice of $\lambda \in (0, 1)$, we have
$\liminf_{x \to \infty} \phi_\infty(x) \geq 1$. On the other hand, we
have $\limsup_{x \to \infty} \phi_\infty(x) \leq \limsup_{x \to \infty} \phi_{n(n-1)}(x) = 1$.
We have proven that $\phi_\infty$ solves \eqref{eqPrescribedScalar} and
has $\lim_{x \to \infty} \phi_\infty(x) = 1$. From the strong maximum
principle, we also conclude that $\phi_\infty > 0$. Hence, we have
proven that Equation \eqref{eqPrescribedScalar} has a solution $\phi_\infty$
such that $\phi_\infty-1 = o(1)$. From Proposition \ref{propImprovedDecay},
we conclude that $\phi_\infty-1 \in X^{2, p}_\delta(M, \bR)$.\\

We still have to prove the existence of the functions $\phibar_0$ and
$\phibar_1$. For small enough $\rho_0 > 0$, we claim that the quadratic
form
\[
 H: u \mapsto \int_{M_{\rho_0}} \left[\frac{4(n-1)}{n-2} |du|^2 + (\scal-\hscal) u^2 \right] d\mu^g
\]
is well defined, continuous and coercive on
\[
 \Wtil^{1, 2}_0(M_{\rho_0}, \bR) \definedas \{u \in W^{1, 2}_0(M_{\rho_0}, \bR),~\tr\vert_{\Sigma_{\rho_0}}(u) = 0\}
\]
Definiteness and
continuity of $H$ follow immediately from Lemma \ref{lmMapping}
applied to $f = \scal - \hscal \in X^{0, p}_\delta(M, \bR)$.
To prove coercivity, we use Formula \eqref{eqPoincare} together with
Lemma \ref{lmMapping}. Choosing $\rho_0$ so that the function $\epsilon_0$
appearing in Formula \eqref{eqPoincare} is greater than or equal to
$-\frac{(n-1)^2}{8}$ on $M_{\rho_0}$, we have, for all
$u \in \Wtil^{1, 2}_0(M_{\rho_0}, \bR)$ and arbitrary $\mu > 0$,
\begin{align*}
H(u)
 &= \int_{M_{\rho_0}} \left[\frac{4(n-1)}{n-2} |du|^2 + (\scal-\hscal) u^2 \right] d\mu^g\\
 &\geq \frac{2(n-1)}{n-2}\int_{M_{\rho_0}} \left[|du|^2 + \left(\frac{(n-1)^2}{4} +\epsilon_0\right) u^2\right] d\mu^g - \int_{M_{\rho_0}} \left|\scal-\hscal\right| u^2 d\mu^g\\
 &\geq \frac{2(n-1)}{n-2}\int_{M_{\rho_0}} \left[|du|^2 + \frac{(n-1)^2}{8} u^2\right] d\mu^g\\
 &\qquad\qquad - \rho_0^\delta \|\scal-\hscal\|_{X^{0,p}_\delta(M, \bR)} \left(\mu^2 \|u\|^2_{W^{1, 2}_0(M_{\rho_0}, \bR)} + C_\mu \|u\|^2_{L^2(M_{\rho_0}, \bR)}\right)\\
 &\gtrsim \|u\|^2_{W^{1, 2}_0(M_{\rho_0}, \bR)}
\end{align*}
where the last inequality holds upon possibly reducing the value of
$\rho_0$. Let $\chibar_{\rho_0/2}$ be the cutoff function defined in
the notation. We solve the equation for $\phibar_0$ by setting
$\phibar_0 = \chibar_{\rho_0/2} + u$, where
$u \in X^{2, p}_{\delta}(M_{\rho_0}, \bR)$ solves
\[
 -\frac{4(n-1)}{n-2} \Delta u + (\scal-\hscal) u = \frac{4(n-1)}{n-2} \Delta \chibar_{\rho_0/2} - (\scal-\hscal) \chibar_{\rho_0/2}.
\]
Note that the right hand side belongs to $X^{0, p}_\delta(M_{\rho_0}, \bR)$
and that this equation is a compact perturbation of the Poisson equation
(see Lemma \ref{lmCompact}). By an extension of the results in
\cite[Appendix A]{GicquaudSakovich} to the case of asymptotically hyperbolic
manifolds with an inner boundary we get that this
equation admits a (unique) solution $u \in X^{2, p}_{\delta}(M_{\rho_0}, \bR)$
provided that the only (weak) solution to
\begin{equation}\label{eq2}
 -\frac{4(n-1)}{n-2} \Delta v + (\scal-\hscal) v = 0
\end{equation}
in $\Wtil^{1, 2}_0(M_{\rho_0}, \bR)$ is $v \equiv 0$. Proving this
extension being lengthy and not much difficult, we leave it as an
exercise to the interested reader. See e.g. \cite[Section 6]{Gicquaud}
for similar results. Assuming that
$v \in \Wtil^{1, 2}_0(M_{\rho_0}, \bR)$ is a weak solution to \eqref{eq2},
we have
\[
 0 = \int_{M_{\rho_0}} \left(-\frac{4(n-1)}{n-2} \Delta v + (\scal-\hscal) v\right) v d\mu^g = H(v).
\]
Since $H$ is coercive, this forces $v \equiv 0$.
\end{proof}

As a consequence,
\[
\boxed{\text{From now on, we will assume that }\hscal \in L^\infty(M, \bR).}
\]

\begin{step}[Reduction to $\hscal-\scal$ compactly supported]\label{stepReduction2}
There exists a function $\phi_0 > 0$,
$\phi_0 - 1 \in X^{2, p}_\delta(M, \bR)$, such that the scalar curvature
$\scaltil$ of the metric $\gtil = \phi_0^{N-2} g$ agrees with $\hscal$
outside some compact set.
\end{step}

\begin{proof}
The argument is based on the implicit function theorem. We choose an
arbitrary $\delta' \in (0, \delta)$ and set
\[
 \scaltil_{\rho_1} = \chibar_{\rho_1} \hscal + (1-\chibar_{\rho_1}) \scal = \scal + \chibar_{\rho_1} \left(\hscal - \scal\right)
\]
for $\rho_1 > 0$ small enough. Then we have
\[
 \left\|\scaltil_{\rho_1} - \scal\right\|_{X^{0, p}_{\delta'}(M, \bR)} = O(\rho_1^{\delta-\delta'}).
\]
As a consequence, $\scaltil_{\rho_1}$ converges to $\scal$ when $\rho_1$
goes to zero, so we set $\scaltil_0 \equiv \scal$. We claim that the
mapping
\begin{equation}\label{eqMappingP}
 P: (u, \rho_1) \mapsto -\frac{4(n-1)}{n-2} \Delta u + \scal~(1+u) - \scaltil_{\rho_1} (1+u)^{N-1}
\end{equation}
is well defined and continuous as a mapping from a neighborhood of the
origin in $X^{2, p}_{\delta'}(M, \bR) \times [0, \infty)$ to
$X^{0, p}_{\delta'}(M, \bR)$ and is differentiable with respect to $u$.
To this end, we rewrite $P(u, \rho_1)$ as follows:
\[
 P(u, \rho_1) = -\frac{4(n-1)}{n-2}\Delta u + \scal~u - \scaltil_{\rho_1} \left((1+u)^{N-1}-1\right) + \scal-\scaltil_{\rho_1}.
\]
Since $u \in X^{2, p}_{\delta'}(M, \bR)$, $\scal \in L^\infty(M, \bR)$ and
$\scal-\scaltil_{\rho_1} \in X^{0, p}_{\delta'}(M, \bR)$, all terms but
the third one in the definition of $P$ are clearly well defined. Note also
that, due to our choice for $p$, we have
$X^{2, p}_{\delta'} \hookrightarrow X^{0, \infty}_{\delta'}(M, \bR)$.
So, due to the multiplication properties of $X$-spaces (see
\cite{GicquaudSakovich}), we have
$\scaltil_{\rho_1} \left((1+u)^{N-1}-1\right) = \scaltil_{\rho_1} u \theta(u) \in X^{0, p}_{\delta'}(M, \bR)$,
where $\theta$ defined as $\theta(u) = ((1+u)^{N-1}-1)/u$ is analytic
over $(-1, \infty)$. Differentiability of $P$ with respect to $u$ and
continuity with respect to $\rho_1$ follow from similar considerations.
Now note that the differential of $P$ with respect to $u$ at $u \equiv 0$
and $\rho_1 = 0$ is given by
\begin{align*}
D_{u}P_{(0, 0)}(v)
 &= -\frac{4(n-1)}{n-2} \Delta v - (N-2)\scal~v
 &= \frac{4(n-1)}{n-2} \left(- \Delta v - \frac{\scal}{n-1}v\right).
\end{align*}
It follows from \cite[Appendix A]{GicquaudSakovich} that $D_{u}P_{(0, 0)}$
is Fredholm with index zero as a mapping from $X^{2, p}_{\delta'}(M, \bR)$
to $X^{0, p}_{\delta'}(M, \bR)$ for any $\delta' \in (-1, n)$
\footnote{The calculation of the indicial radius follows from the fact
that $D_{u}P_{(0, 0)}$ is (up to a multiplicative constant) a compact
perturbation of $v \mapsto -\Delta v + nv$ (see Lemma \ref{lmCompact})
whose indicial radius is $\frac{n+1}{2}$ (see \cite[Corollary 7.4]{LeeFredholm}).}.
Hence, $D_{u}P_{(0, 0)}$ will be an isomorphism provided  that the
$L^2$-kernel of $D_{u}P_{(0, 0)}$ is reduced to $\{0\}$. This is not
expected to hold in general. But, assuming that $\scal \leq 0$, if
$v \in W^{1, 2}_0(M, \bR)$ solves
\[
 - \Delta v - \frac{\scal}{n-1}v = 0,
\]
we have,
\[
 \int_M \left( |dv|^2 + \frac{\scal}{n-1} v^2\right) d\mu^g = 0 ~\Rightarrow v \equiv 0.
\]
The assumption $\scal \leq 0$ can be fulfilled at the costless price of
replacing the metric $g$ by some well chosen metric $g'$ conformal to $g$
(see \cite{AvilesMcowen,AnderssonChruscielFriedrich,Gicquaud}).
By the implicit function theorem, we get that the equation $P(u, \rho_1) = 0$
has a solution $u \in X^{2, p}_{\delta'}(M, \bR)$ for sufficiently small
$\rho_1$. Setting $\phi_0 = 1+u$, the scalar curvature of the metric
$\phi_0^{N-2} g$ agrees with $\hscal$ on $M_{\rho_1}$. Proposition
\ref{propImprovedDecay} ensures that $u = \phi_0-1 \in X^{2, p}_\delta(M,  \bR)$.
\end{proof}

We shall now work with the metric $\gtil$ as a background metric.
In particular, Lebesgue and Sobolev norms will be defined using the
metric $\gtil$. There is a caveat: the metric $\gtil$ is not asymptotically
hyperbolic in the sense we gave in Section \ref{secPrelim}! In particular,
Lemma \ref{lmPoincare} and Proposition \ref{propSemiContinuity} must be
reproven for the metric $\gtil$:

\begin{lemma}\label{lmPoincare2}
There exists a function $\widetilde{\epsilon} = \widetilde{\epsilon}(x) \in C^0(M, \bR)$,
tending to zero at infinity, such that for all $u \in W^{1, 2}_0(M, \bR)$,
we have
\[
 \int_M \left[\frac{4(n-1)}{n-2} |du|_{\gtil}^2 + \scaltil~u^2\right]d\mu^{\gtil}
 \geq
 \int_M \left(\frac{n-1}{n-2} + \widetilde{\epsilon}\right)u^2 d\mu^{\gtil}.
\]
\end{lemma}

\begin{proof}
This follows at once from the calculation in the proof of Proposition
\ref{propConfInvariance}: for all $u \in W^{1, 2}_0(M, \bR)$, we have
\begin{align*}
G_{\gtil}(u)
 &= G_g(\phi_0 u)\\
 &\geq \int_M \left(\frac{n-1}{n-2} + \epsilon\right) (\phi_0u)^2 d\mu^g\\
 &\geq \int_M \left(\frac{n-1}{n-2} + \epsilon\right) \phi_0^{2-N} u^2 d\mu^{\gtil}.
\end{align*}
Since $\phi_0 - 1 \in X^{2, p}_\delta(M, \bR) \subset X^{0, \infty}_\delta(M, \bR)$,
we have
\[
 \left(\frac{n-1}{n-2} + \epsilon\right) \phi_0^{2-N} = \frac{n-1}{n-2} + \widetilde{\epsilon},
\]
with $\widetilde{\epsilon} = o(1)$. This ends the proof of the lemma.
\end{proof}

\begin{lemma}\label{lmSemiContinuity2}
The functional $G_{\gtil}$
is sequentially weakly lower semi-continuous on $W^{1, 2}_0(M, \bR)$.
\end{lemma}

\begin{proof}
From the proof of Proposition \ref{propConfInvariance}, multiplication by
$\phi_0$ is a bounded operator on $W^{1, 2}_0(M, \bR)$. As a consequence,
if $(u_k)_k$ is a sequence of elements of $W^{1, 2}_0(M, \bR)$ converging
weakly to some $u_\infty$, we have that
$\phi_0 u_k \hookrightarrow \phi_0 u_\infty$.
As a consequence,
\[
 \liminf_{k\to \infty} G_{\gtil}(u_k) = \liminf_{k\to \infty} G_g(\phi_0 u_k) \geq G_g(\phi_0 u_\infty) = G_{\gtil}(u_\infty).
\]
\end{proof}

We also note that the argument in \cite{DiltsMaxwell} can be simplified
at this point.
We introduce the following functional defined on $W^{1, 2}_0(M, \bR)$
\begin{equation}\label{edDefF}
 F(u) \definedas \int_M \left[\frac{4(n-1)}{n-2} |du|_{\gtil}^2 + \scaltil \left((u+1)^2-1\right) - \frac{2}{N} \hscal \left(|u+1|^N-1\right)\right] d\mu^{\gtil}.
\end{equation}

\begin{step}\label{stepContinuity}
The functional $F$ is well-defined on $W^{1, 2}_0(M, \bR)$ and continuous
for the strong topology.
\end{step}

\begin{proof}
To prove that $F$ is well-defined, we rewrite it as follows:
\begin{equation}\label{eqRedefF}
 \begin{aligned}
  F(u) &= \int_M \left[\frac{4(n-1)}{n-2} |du|_{\gtil}^2 + \left(\scaltil - \hscal\right) \left((u+1)^2-1\right)\right.\\
   &\qquad\qquad\left. - \hscal \left(\frac{2}{N} |u+1|^N - (u+1)^2 + 1 - \frac{2}{N}\right)\right] d\mu^{\gtil}.
 \end{aligned}
\end{equation}
Since $\scaltil - \hscal$ is bounded with compact support, the second term
in the integral can be estimated using the Cauchy-Schwarz inequality by
noticing that $(u+1)^2-1 = u^2 + 2u$. The only difficult term is the last
one. Some simple calculation shows, however, that
\[
 0 \leq \frac{2}{N} |u+1|^N - (u+1)^2 + 1 - \frac{2}{N} \leq \alpha u^2 + \beta |u|^N
\]
for some well chosen $\alpha > N-2$ and $\beta > \frac{2}{N}$. Hence, since
$\hscal \in L^\infty(M, \bR)$, the last term is well defined by the Sobolev embedding
theorem. Next we prove strong continuity of $F$. Note that
\[
 u \mapsto \int_M \left[\frac{4(n-1)}{n-2} |du|_{\gtil}^2 + \left(\scaltil - \hscal\right) ((u+1)^2-1)\right] d\mu^{\gtil}
\]
is clearly continuous as the sum of a bounded quadratic form and a
bounded linear form. The only problem comes from the last term in
\eqref{eqRedefF}. Note that the function
\begin{equation}\label{eqDefh}
 h(x) \definedas \frac{2}{N} \left[|x+1|^N-1\right] - (x+1)^2 + 1
\end{equation}
satisfies
\[
 h'(x) = 2 \left(|x+1|^{N-2}-1\right)(x+1).
\]
Hence,
\[
 |h'(x)| \leq \alpha' |x|^{N-1} + \beta' |x|
\]
for some $\alpha', \beta' \geq 0$.
So, on any interval $[a, b]$ we have
\begin{align*}
\left|h(b) - h(a)\right|
 &\leq |b-a| \sup_{x \in [a, b]} |h'(x)|\\
 &\leq |b-a| \left[\alpha' \left(|a|^{N-1}+|b|^{N-1}\right) + \beta' (|a|+|b|)\right].
\end{align*}

As a consequence, given $u, v \in W^{1, 2}_0(M, \bR)$, we have
\begin{align*}
 &\left|\int_M \left(- \hscal\right) \left(\frac{2}{N} |u+1|^N - (u+1)^2 + 1 - \frac{2}{N}\right) d\mu^{\gtil} - \int_M \left(- \hscal\right) \left(\frac{2}{N} |v+1|^N - (v+1)^2 + 1 - \frac{2}{N}\right) d\mu^{\gtil}\right|\\
 &\qquad = \left|\int_M \left(- \hscal\right) \left(h(u) - h(v)\right] d\mu^{\gtil}\right|\\
 &\qquad \leq \int_M \left(- \hscal\right)  \left|h(u) - h(v)\right| d\mu^{\gtil}\\
 &\qquad \leq \int_M \left(- \hscal\right) |u-v| \left[\alpha' \left(|u|^{N-1}+|v|^{N-1}\right) + \beta' (|u|+|v|)\right] d\mu^{\gtil}\\
 &\qquad \leq \left\|\hscal\right\|_{L^\infty(M, \bR)} \left[\alpha' (\|u\|_{L^N(M, \bR)}^{N-1} + \|v\|_{L^N(M, \bR)}^{N-1}) \|u-v\|_{L^N(M, \bR)}\right.\\
 &\qquad\qquad\left. + \beta'(\|u\|_{L^2(M, \bR)} + \|v\|_{L^2(M, \bR)}) \|u-v\|_{L^2(M, \bR)}\right]\\
 &\qquad \leq C\left(\left\|\hscal\right\|_{L^\infty(M, \bR)}, \|u\|_{W^{1, 2}_0(M, \bR)}, \|v\|_{W^{1, 2}_0(M, \bR)}\right) \|u-v\|_{W^{1,2}(M, \bR)},
\end{align*}
where the constant $C$ depends continuously on $(\|u\|_{W^{1, 2}_0(M, \bR)}, \|v\|_{W^{1, 2}_0(M, \bR)})$.
This shows that $F$ is locally Lipschitz continuous.
\end{proof}

Our next goal is to prove that $F$ is coercive.
Before that, we need a lemma adapted from \cite{Rauzy}
and \cite[Proposition 4.5]{DiltsMaxwell}:

\begin{lemma}\label{lmIntegration}
Assuming that $\cZ$ has $\cY_g(\cZ) > 0$, there exist positive constants
$\eta$ and $\epsilon$ such that for all $u \in W^{1, 2}_0(M, \bR)$,
we have
\[
 \int_M \left|\hscal\right|~u^N d\mu^{\gtil} \leq \eta \|u\|^N_{L^N(M, \bR)} \Rightarrow G_{\gtil}(u) \geq \epsilon \|u\|^2_{L^2(M, \bR)}.
\]
\end{lemma}

\begin{proof}
The argument goes by contradiction. We assume that there exists a sequence of
functions $u_k \in W^{1, 2}_0(M, \bR)$, $u_k \not\equiv 0$ a.e., such that
\[
 \frac{1}{\|u_k\|_{L^N(M, \bR)}^N} \int_M \left|\hscal\right|~u_k^N d\mu^{\gtil} \to 0
 \quad\text{and}\quad
 \frac{G_{\gtil}(u_k)}{\|u_k\|_{L^2(M, \bR)}^2}\to 0.
\]
We argue as in the proof of Proposition \ref{propYamabe} using Lemma
\ref{lmPoincare2} instead of Lemma \ref{lmPoincare}. There exist
a compact subset $K \ssubset M$ and a constant $C > 0$ such that
\[
 \|u\|_{L^2(M, \bR)}^2 \leq G_{\gtil}(u) + (C+1) \|u\|_{L^2(K, \bR)}^2.
\]
Letting
\[
 \epsilon_k \definedas \frac{G_{\gtil}(u_k)}{\|u_k\|_{L^2(M, \bR)}^2} \to 0,
\]
we have that
\[
 \|u_k\|_{L^2(M, \bR)}^2 \leq \epsilon_k \|u_k\|_{L^2(M, \bR)}^2 + (C+1) \|u\|_{L^2(K, \bR)}^2.
\]
If $k$ is large enough so that $\epsilon_k \leq 1/2$, we obtain
\[
 \|u_k\|_{L^2(M, \bR)}^2 \leq 2(C+1) \|u_k\|_{L^2(K, \bR)}^2.
\]
Since
\[
 \frac{4(n-1)}{n-2} \|du_k\|^2_{L^2(M, \bR)} - \left\|\scaltil\right\|_{L^\infty(M, \bR)} \|u_k\|_{L^2(M, \bR)}^2 \leq G_{\gtil}(u_k) = \epsilon_k \|u_k\|_{L^2(M, \bR)}^2
\]
we have that (provided $k$ is large enough)
\[
 \|u_k\|_{W^{1, 2}_0(M, \bR)} \lesssim \|u_k\|_{L^2(M, \bR)} \lesssim \|u_k\|_{L^2(K, \bR)}.
\]
We can assume that $\|u_k\|_{L^2(K, \bR)} = 1$ so that the sequence
$(u_k)_{k \geq 0}$ is bounded in $W^{1, 2}_0(M, \bR)$. By weak compactness,
we can asume further that there exists a function $u_\infty \in W^{1, 2}_0(M, \bR)$
such that $(u_k)_{k \geq 0}$ converges weakly (in $W^{1, 2}_0(M, \bR)$) to
$u_\infty$ and strongly in $L^2_{\mathrm{loc}}(M, \bR)$. We have $u_\infty \not\equiv 0$
since $\|u_\infty\|_{L^2(K, \bR)} = 1$. The function
\[
 u \mapsto \int_M |\hscal| u^N d\mu^{\gtil}
\]
is strongly continuous on $W^{1, 2}_0(M, \bR)$ and convex. Hence it is weakly
lower semicontinuous:
\[
 \int_M |\hscal| u_\infty^N d\mu^{\gtil} \leq \liminf_{k \to \infty} \int_M |\hscal| u_k^N d\mu^{\gtil} = 0.
\]
This shows that $u_\infty \equiv 0$ a.e. on $M \setminus \cZ$, i.e.
$u_\infty \in \cF(\cZ)$. We now get a contradiction since, by the assumption
made for $\cZ$, we have $G_{\gtil}(u_\infty) > 0$ while, by the lower semicontinuity
of $G_{\gtil}$ (Lemma \ref{lmSemiContinuity2}), we have
$G_{\gtil}(u_\infty) \leq \liminf_{k \to \infty} G_{\gtil}(u_k) = 0$.
\end{proof}

\begin{step}\label{stepCoercivity}
Assuming that $\cY_g(Z) > 0$, the functional $F$ is coercive, meaning
that for all $A > 0$ there exists a $B > 0$ such that
\[
 \forall u \in W^{1, 2}_0(M, \bR),~F(u) \leq A \Rightarrow \|u\|_{W^{1, 2}_0(M, \bR)} \leq B.
\]
\end{step}

\begin{proof}
We assume by contradiction that there exists $A > 0$ and a sequence
$(u_k)_k$ of elements of $W^{1, 2}_0(M, \bR)$ such that $F(u_k) \leq A$
while $\|u_k\|_{W^{1, 2}_0(M, \bR)} \to \infty$. We rewrite $F(u_k)$ using
formula \eqref{eqRedefF}:
\begin{align*}
 F(u_k) &= \int_M \left[\frac{4(n-1)}{n-2} |du_k|_{\gtil}^2 + \left(\scal - \hscal\right) \left((u_k+1)^2-1\right)\right.\\
   &\qquad\qquad\left. - \hscal \left(\frac{2}{N} |u_k+1|^N - (u_k+1)^2 + 1 - \frac{2}{N}\right)\right] d\mu^{\gtil}.
\end{align*}

Note that the function
\begin{equation}\label{eqDefH}
 h: x \mapsto \frac{2}{N} |u_k+1|^N - (u_k+1)^2 + 1 - \frac{2}{N}
\end{equation}
is non-negative over $\bR$ (this follows by studying its variations over
$[-1, \infty)$ and using parity). Since $\hscal \leq 0$, we have
\begin{equation}\label{eq0}
 \int_M \left[\frac{4(n-1)}{n-2} |du_k|^2 + \left(\scal - \hscal\right) \left((u_k+1)^2-1\right)\right]d\mu^{\gtil} \leq A.
\end{equation}
In particular, letting $\Omega$ be any open set outside which
$\scaltil \equiv \hscal$, if
$\|u_k\|_{L^2(\Omega, \bR)}$ is bounded, we conclude that
$\|u_k\|_{W^{1, 2}_0(M, \bR)}$ is also bounded which contradicts the
assumption. As a consequence, $\|u_k\|_{L^2(\Omega, \bR)}$ is unbounded.

Assume now that for an infinite number of values of $k$, we have
\[
 \int_M \left|\hscal\right|~u_k^N d\mu^{\gtil} \leq \eta \|u_k\|^N_{L^N(M, \bR)},
\]
where $\eta$ has been defined in Lemma \ref{lmIntegration}. Then we
have for all such $k$,
\[
 G_{\gtil}(u_k) \geq \epsilon \|u_k\|^2_{L^2(M, \bR)} \to \infty.
\]
We rewrite
\[
 F(u_k) = G_{\gtil}(u_k) + 2 \int_M \left(\scaltil-\hscal\right) u_k d\mu^{\gtil} - \frac{2}{N} \int_M \hscal~\left[|u_k+1|^N - N u_k - 1\right]d\mu^{\gtil}.
\]
By the convexity of $x \mapsto |x+1|^N$, we have that
$|u_k+1|^N - N u_k - 1 \geq 0$ so
\begin{align*}
F(u_k)
 &\geq G_{\gtil}(u_k) + 2 \int_M \left(\scal-\hscal\right) u_k d\mu^{\gtil}\\
 &\geq \epsilon \|u_k\|^2_{L^2(M, \bR)} - 2 \left\|\scaltil - \hscal\right\|_{L^2(M, \bR)} \|u_k\|_{L^2(M, \bR)}.
\end{align*}
As a consequence, for $k$ large enough, we have $F(u_k) > A$. This
contradicts our assumption. So, for $k$ large enough, we have
\[
 \int_M \left|\hscal\right|~u_k^N d\mu^{\gtil} \geq \eta \|u_k\|^N_{L^N(M, \bR)}.
\]
We can write
$h(x) = \frac{2}{N} |x|^N + f(x)$, where $|f(x)| \leq C (|x|^{N-1} + |x|^2)$
for some constant $C = C(n )> 0$, where $h$ was defined in \eqref{eqDefH},
so, from Equation \eqref{eqRedefF}, we have
\begin{align*}
&F(u_k)\\
&= \int_M \left[\frac{4(n-1)}{n-2} |du_k|_{\gtil}^2 + \left(\scaltil - \hscal\right) \left((u_k+1)^2-1\right) - \hscal \left(\frac{2}{N} |u_k|^N + f(u_k)\right)\right] d\mu^{\gtil}\\
&\geq \int_M \left[\frac{4(n-1)}{n-2} |du_k|_{\gtil}^2 + \left(\scaltil - \hscal\right) \left((u_k+1)^2-1\right) - \frac{2}{N} \hscal |u_k|^N\right] d\mu^{\gtil}\\
&\qquad\qquad - C\left\|\hscal\right\|_{L^\infty(M, \bR)} \left(\|u_k\|_{L^{N-1}(M, \bR)}^{N-1} + \|u_k\|_{L^2(M, \bR)}^2\right)\\
&\geq \int_M \left[\frac{4(n-1)}{n-2} |du_k|_{\gtil}^2 + \left(\scaltil - \hscal\right) \left((u_k+1)^2-1\right)\right] d\mu^{\gtil}\\
&\qquad\qquad + \frac{2\eta}{N} \|u_k\|_{L^N(M, \bR)}^N - C\left\|\hscal\right\|_{L^\infty(M, \bR)} \left(\|u_k\|_{L^N(M, \bR)}^{\lambda(N-1)} \|u_k\|_{L^2(M, \bR)}^{(1-\lambda)(N-1)} + \|u_k\|_{L^2(M, \bR)}^2\right),
\end{align*}
where $\lambda \in (0, 1)$ is such that
\[
 \frac{1}{N-1} = \frac{\lambda}{N} + \frac{1-\lambda}{2}.
\]
From Equation \eqref{eq0}, we see that, for some large constant $\Lambda$,
\[
 \|u_k\|^2_{W^{1, 2}_0(M, \bR)} \leq \Lambda \left(\|u_k\|_{L^2(\Omega, \bR)}^2 + 1\right).
\]
Since $\Omega$ is compact, we have
\[
 \|u_k\|_{L^2(\Omega, \bR)}^2 \leq \|u_k\|_{L^N(\Omega, \bR)}^2 \vol_g(\Omega)^n \leq \|u_k\|_{L^N(M, \bR)}^2 \vol_g(\Omega)^n.
\]
In particular, $\|u_k\|_{L^N(M, \bR)} \to \infty$. We finally arrive at
the following asymptotic inequality
\[
 F(u_k) \geq \frac{2\eta}{N} \|u_k\|_{L^N(M, \bR)}^N + O\left(\|u_k\|_{L^N(M, \bR)}^{N-1} + \|u_k\|_{L^N(M, \bR)}^2\right),
\]
which yields once again a contradiction.
This ends the proof of the coercivity of $F$.
\end{proof}

\begin{step}\label{stepContinuity2}
The functional $F$ is sequentially lower semicontinuous for the weak
topology on $W^{1, 2}_0(M, \bR)$.
\end{step}

\begin{proof}
This is a simple calculation. We rewrite $F(u)$ as follows:
\begin{align*}
F(u)
 &= \int_M \left[\frac{4(n-1)}{n-2} |du|_{\gtil}^2 + \scaltil~u^2 - \frac{2}{N}\hscal \left(|u+1|^N - 1 - Nu\right) + 2 \left(\scal - \hscal\right) u\right] d\mu^{\gtil}\\
 &= G_{\gtil}(u) + \frac{2}{N} \int_M \left(- \hscal\right) \left(|u+1|^N - 1 - Nu\right) d\mu^{\gtil} + 2 \int_M \left(\scaltil - \hscal\right) u d\mu^{\gtil}
\end{align*}
From Lemma \ref{lmSemiContinuity2}, we have that $u \mapsto G_{\gtil}(u)$
is sequentially lower semicontinuous. The remaining two terms are clearly
strongly continuous and convex so, in particular, weakly lower semicontinuous.
\end{proof}

\begin{step}\label{stepExistence}
There exists a minimizer $u \in W^{1, 2}_0(M, \bR)$ for $F$ on $\cF$. The
function $\phi \definedas \phi_0(1+u)$ is positive and solves the prescribed
scalar curvature equation \eqref{eqPrescribedScalar}.
Hence, the metric $\ghat = \phi^{N-2} g$ has scalar curvature $\hscal$.
\end{step}

\begin{proof}
Existence of a minimizer $u$ follows at once from Step \ref{stepContinuity2}
and Step \ref{stepCoercivity}. Upon replacing $u$ by $\ubar = |u+1|-1$
which satisfies $F(\ubar) = F(u)$ so $\ubar$ is another minimizer for $F$,
we can assume that $u \geq -1$ i.e. $\phi \geq 0$. We are left to show that
$\phi$ is positive. By standard elliptic regularity, we have that
$\phi \in W^{2, p}_{loc}(M, \bR) \subset L^\infty_{loc}(M, \bR)$.
From the Harnack inequality \cite[Theorems 1.1 and 5.1]{TrudingerHarnack},
we conclude that $\phi > 0$.
\end{proof}

\begin{step}\label{stepUniqueness}
The function $\phi$ constructed in Step \ref{stepExistence} belongs to
$1+X^{2, p}_\delta(M, \bR)$ for all $\delta \in (0, n)$.
\end{step}

This follows from (local) elliptic regularity and Proposition
\ref{propImprovedDecay}.

\section{Solutions to the Lichnerowicz equation}\label{secLich}
In this section, we consider the Lichnerowicz equation
\begin{equation}\label{eqLichnerowicz}
 -\frac{4(n-1)}{n-2} \Delta \phi + \scal~\phi + \frac{n-1}{n} \tau^2 \phi^{N-1} = \frac{A^2}{\phi^{N+1}},
\end{equation}
where the unknown is the positive function $\phi$, and $\tau, A$ are two
given functions. We refer the reader to \cite{BartnikIsenberg} for an
introduction to this equation and to \cite{AnderssonChrusciel,Gicquaud,Sakovich}
for a detailed study of this equation in the asymptotically hyperbolic
setting. When $A \equiv 0$ this equation reduces to
\begin{equation}\label{eqLichnerowicz0}
 -\frac{4(n-1)}{n-2} \Delta \phi + \scal~\phi + \frac{n-1}{n} \tau^2 \phi^{N-1} = 0
\end{equation}
which is nothing but \eqref{eqPrescribedScalar} with
$\hscal = -\frac{n-1}{n}\tau^2$.
We first state and prove a result regarding the monotonicity method to
solve semilinear elliptic PDEs on asymptotically hyperbolic manifolds and
in a low regularity context:

\begin{proposition}\label{propMonotonicity}
Let $F: M \times (0, \infty) \to \bR$ be a function of the form
\[
 F(x, \phi) = \sum_{i \in I} a_i(x) \phi^{\lambda_i},
\]
where $I$ is a finite set, $a_i \in X^{0, p}_0(M, \bR)$ and
$\lambda_i \in \bR$. Assume that there exist two positive functions
$\phi_{\pm} \in X^{2, p}_0(M, \bR)$ such that $\phi_- \leq \phi_+$
with $\phi_-$ bounded from below and
\[
 -\Delta \phi_+ + F(x, \phi_+) \geq 0 \text{ a.e. on } M \text{ and } -\Delta \phi_- + F(x, \phi_-) \leq 0 \text{ a.e. on } M.
\]
Then there exists a function $\phi \in X^{2, p}_0(M, \bR)$,
$\phi_- \leq \phi \leq \phi_+$, such that
\begin{equation}\label{eqPhi}
 -\Delta \phi + F(x, \phi) = 0.
\end{equation}
\end{proposition}

\begin{proof}
The proof is standard in a context with more regularity (see e.g.
\cite{AnderssonChrusciel} for a proof based on the Perron method
and \cite{Gicquaud} for the construction of a monotone sequence
of functions). The proof we adapt here is taken fron \cite[Chapter 14]{Taylor3}.
We choose a function $\omega \in X^{0, p}_0$ such that the function
$\phi \mapsto \omega(x) \phi - F(x, \phi)$ is well-defined and increasing
for almost every $x \in M$ on the interval $[\phi_-(x), \phi_+(x)]$. This
can be done as follows. Let
\[
 H(x, \phi) = \omega(x) \phi - F(x, \phi),
\]
then
\begin{align*}
\pdiff{H}{\phi}
 &= \omega(x) - \sum_{i \in I} \lambda_i a_i(x) \phi^{\lambda_i-1}\\
 &\geq \omega(x) - \sum_{i \in I} |\lambda_i| |a_i(x)| \phi^{\lambda_i-1}\\
 &\geq \omega(x) - \sum_{\substack{i \in I\\ \lambda_i \geq 1}} |\lambda_i| |a_i(x)| \phi_+^{\lambda_i-1} - \sum_{\substack{i \in I\\ \lambda_i < 1}} |\lambda_i| |a_i(x)| \phi_-^{\lambda_i-1}
\end{align*}
So the condition on $H$ will be achieved provided we choose
$\omega \in X^{0, p}_0(M, \bR)$ such that
\[
 \omega(x) \geq \sum_{\substack{i \in I\\ \lambda_i \geq 1}} |\lambda_i| |a_i(x)| \phi_+^{\lambda_i-1}
 + \sum_{\substack{i \in I\\ \lambda_i < 1}} |\lambda_i| |a_i(x)| \phi_-^{\lambda_i-1} \in X^{0, p}_{\delta}(M, \bR),
\]
(note that the right hand side belonqs to $X^{0, p}_\delta(M, \bR)$
because we assumed that $\phi_-$ is bounded from below and
$\phi_+ \in X^{2, p}_\delta(M, \bR) \subset L^\infty(M, \bR)$).
Let $(\Omega_k)_k$ be a nested sequence of bounded non-empty open
subsets of $M$ with smooth boundary such that $M = \bigcup_k \Omega_k$.
For each $k$, let
\[
 \cC_k = \{\phi \in L^\infty(\Omega_k, \bR),~\phi_- \leq \phi \leq \phi_+\}.
\]
For any $\phi \in \cC_k$, let $\psi = \cF_k(\phi) \in W^{2, p}(\Omega_k, \bR)$
denote the solution to
\begin{equation}\label{eqPhi2}
-\Delta \psi + \omega \psi = \omega \phi - F(x, \phi),\quad
\psi = \phi_- \quad\text{on } \partial \Omega_k.
\end{equation}
(the solution exists and is unique due to the fact that $\omega \geq 0$).
We claim that $\psi \in \cC_k$. Indeed, since $H$ is increasing, we have
\[
 -\Delta \psi + \omega \psi = H(x, \phi) \geq H(x, \phi_-),\quad \psi = \phi_-\quad\text{on } \partial \Omega_k.
\]
So, from the maximum principle, we conclude that $\psi \geq \phi_-$.
Similarly, we have $\psi \leq \phi_+$. Thus $\psi \in \cC_k$.

Now remark that $\cC_k$ is a convex subset of $L^\infty(M, \bR)$. It
can be easily seen that $\cF_k$ is a continuous mapping and, due to the
fact that  $W^{2, p}(\omega_k, \bR)$ embeds compactly into
$L^\infty(\Omega_k, \bR)$, the image of $\cC_k$ is relatively compact.
It follows from Schauder's fixed point theorem that there exists
at least one solution $\phi_k \in W^{2, p}(M, \bR)$, $\phi_k \in \cC_k$
to the following problem
\begin{equation}\label{eqPhi3}
-\Delta \phi_k + \omega \phi_k = \omega \phi_k - F(x, \phi_k),\quad
\phi_k = \phi_- \quad\text{on } \partial \Omega_k.
\end{equation}
We select one such solution randomly for each $k$.

If $U$ and $V$ are bounded open subsets of $M$, $U \subset\subset V$
there exists $k_0 \geq 0$ such that $V \subset \Omega_k$ for all
$k \geq k_0$ (this is due to the fact that $\Vbar$ is compact).
Since, for all $k \geq k_0$, we have $\phi_- \leq \phi_k \leq \phi_+$,
the sequence of functions $(F(x, \phi_k))_{k \geq k_0}$ is bounded in
$L^p(V, \bR)$. By elliptic regularity, we conclude that
$(\phi_k)_{k \geq k_0}$ is bounded in $W^{2, p}(U, \bR)$ so, in particular,
there exists a subsequence of $(\phi_k)_{k \geq k_0}$ that converges
in $L^\infty(U, \bR)$. By a diagonal extraction process similar to
Lemma \ref{lmCompact}, we construct a function $\phi$,
$\phi_- \leq \phi \leq \phi_+$, $\phi \in W^{2, p}_{loc}(M, \bR)$
solving \eqref{eqPhi}. Finally, $\phi \in X^{2, p}_0(M, \bR)$ by elliptic
regularity.
\end{proof}

The following result is based on \cite[Theorem 7.1]{Sakovich}.

\begin{theorem}\label{thmLich}
Assume given two non-negative functions $A$ and $\tau$ such that
$A, \tau-n \in X^{0, 2p}_\delta(M, \bR)$, $\delta \in (0, n)$. The
following statements are equivalents:
\begin{enumerate}
\renewcommand{\theenumi}{\roman{enumi}}
\renewcommand{\labelenumi}{\theenumi.}
\item\label{it1} There exists a positive solution $\phi \in 1 + X^{2, p}_\delta(M, \bR)$
to \eqref{eqLichnerowicz},
\item\label{it2} There exists a positive solution $\phitil \in 1 + X^{2, p}_\delta(M, \bR)$
to \eqref{eqLichnerowicz} with $A \equiv 0$,
\item\label{it3} The set $\cZ \definedas \tau^{-1}(0)$ satisfies
$\cY_g(\cZ) > 0$.
\end{enumerate}
Further, when a solution $\phi \in 1 + X^{2, p}_\delta(M, \bR)$ exists to
\eqref{eqLichnerowicz}, it is unique.
\end{theorem}

\begin{proof}
The equivalence between \ref{it2} and \ref{it3} follows from Theorem
\ref{thmA}. To show the equivalence between \ref{it1} and \ref{it2}, we
follow the argument given in \cite{Sakovich}.

Assume first that \ref{it1} holds. Then the solution $\phi$ to
\eqref{eqLichnerowicz} is a supersolution to \eqref{eqLichnerowicz0}.
The zero function being a subsolution to \eqref{eqLichnerowicz0}, we
conclude from Proposition \ref{propMonotonicity} that there exists a
solution $\phitil$ to \eqref{eqLichnerowicz0}, $0 \leq \phitil \leq \phi$.
Yet, we have to be more cautious, we have to rule out the possibility
that $\phitil \equiv 0$. This is achieved by using the lower barrier
$\phi_-$ defined in the proof of Step \ref{stepReduction} in the previous
section and noting that,
by a straightforward induction argument, $\phi_k \geq \phi_-$ for all
$k \geq 0$, where $(\phi_k)_{k \geq 0}$ is the iteration sequence of the
monotonicity method. Thus, $\phitil \not\equiv 0$. So, finally,
$\phitil > 0$ by the Harnack inequality. This show that \ref{it1}
$\Rightarrow$ \ref{it2}.

Conversely, assume that \ref{it2} holds. Then the solution
$\phitil$ to \eqref{eqLichnerowicz0} is a subsolution to
\eqref{eqLichnerowicz}. We still need to construct a supersolution
to \eqref{eqLichnerowicz}. To this end, we perform a conformal change:
Set $\gtil \definedas \phitil^{N-2} g$ and $\Atil = \phitil^{-N} A$,
then \eqref{eqLichnerowicz} becomes
\begin{equation}\label{eqLichnerowicz1}
 -\frac{4(n-1)}{n-2} \Deltatil \frac{\phi}{\phitil} + \frac{n-1}{n} \tau^2 \left[\left(\frac{\phi}{\phitil}\right)^{N-1} - \frac{\phi}{\phitil}\right] = \Atil^2 \left(\frac{\phi}{\phitil}\right)^{-N-1},
\end{equation}
where $\Deltatil$ is the Laplace operator for the metric $\gtil$. Following
\cite{MaxwellNonCMC}, we introduce the following equation for $\util$:
\begin{equation}\label{eqLichnerowiczl}
 -\frac{4(n-1)}{n-2} \Deltatil \util + (N-2)\frac{n-1}{n} \tau^2 \util = \Atil^2.
\end{equation}
This equation can be rewritten as
\begin{equation}\label{eqLichnerowiczl0}
 -\frac{4(n-1)}{n-2} \Delta u + \scal~u + (N-1)\frac{n-1}{n} \tau^2 \phitil^{N-2} u = \Atil^2 \phitil^{N-1},
\end{equation}
where $u = \phitil \util$. From Lemma \ref{lmCompact}, the operator on the
left is a compact perturbation of
\[
 u \mapsto -\frac{4(n-1)}{n-2} \Delta u + \scal~u + \frac{n(n-1)(n+2)}{n-2} u
\]
which is Fredholm with zero index (from \cite[Theorem C]{LeeFredholm}) and
has indicial radius $(n+1)/2$. Further, note that the homogeneous equation
associated to \eqref{eqLichnerowiczl} has no non trivial solution in
$W^{1, 2}_0(M, \bR)$. This
shows that Equation \eqref{eqLichnerowiczl} (or equivalently
\eqref{eqLichnerowiczl0}) admits a unique solution $\util \in X^{2, p}_\delta(M, \bR)$.
Applying the maximum principle to \eqref{eqLichnerowiczl}, we see that
$u \geq 0$. One then readily check that $1+u$ is a supersolution to
\eqref{eqLichnerowicz1} and, hence, that $\phitil(1+u)$ is a supersolution
to \eqref{eqLichnerowicz}. We have found a subsolution $\phitil$ and
a supersolution $\phitil(1+u) \geq \phitil$ to \eqref{eqLichnerowicz}, so,
from Proposition \ref{propMonotonicity}, we get a solution to \eqref{eqLichnerowicz},
proving that \ref{it1} holds. So \ref{it2} $\Rightarrow$ \ref{it1}.

Uniqueness of the solution is proven by means similar to
that of Proposition \ref{propMaximumPrinciple}. We refer the reader to
\cite{AnderssonChrusciel} or \cite[Section 8]{Sakovich}.
\end{proof}

\providecommand{\bysame}{\leavevmode\hbox to3em{\hrulefill}\thinspace}
\providecommand{\MR}{\relax\ifhmode\unskip\space\fi MR }
\providecommand{\MRhref}[2]{%
  \href{http://www.ams.org/mathscinet-getitem?mr=#1}{#2}
}
\providecommand{\href}[2]{#2}

\end{document}